\documentclass[11pt,a4paper,reqno]{amsart}
\usepackage{amsmath,amssymb,amsfonts,epsfig,mathrsfs,cite, hyperref}
\usepackage[T1]{fontenc}
\usepackage{color}
\usepackage{array}
\usepackage{amsthm}
\usepackage{amstext}
\usepackage{graphicx}
\usepackage{setspace}

\usepackage[title]{appendix}

\usepackage{booktabs}
\usepackage{longtable}
\usepackage{tcolorbox}

\usepackage{float}

\makeatletter
\@namedef{subjclassname@2020}{%
  \textup{2020} Mathematics Subject Classification}
\makeatother

\usepackage[margin=2.5cm]{geometry}
\usepackage{color}
\usepackage{enumitem}
\usepackage{amscd,psfrag}
\usepackage{yhmath}
\usepackage[mathscr]{eucal}
\usepackage{pdflscape}

\usepackage{comment}

\allowdisplaybreaks[4]

\usepackage{slashed}

\makeatletter
\pdfpageheight\paperheight
\pdfpagewidth\paperwidth

\usepackage{epstopdf}
\usepackage{indentfirst}	

\usepackage[normalem]{ulem}
\theoremstyle{plain}
\newtheorem{definition}{Definition}
\newtheorem{theorem}[definition]{Theorem}
\newtheorem*{theorem*}{Theorem}

\newtheorem{remark}[definition]{Remark}

\newtheorem*{remark*}{Remark}
\newtheorem*{sideremark*}{Side Remark}

\newtheorem*{claim*}{Claim}
\newtheorem*{lemma*}{Lemma}
\newtheorem*{q*}{Question}
\newtheorem{lemma}[definition]{Lemma}

\newtheorem*{corollary*}{Corollary}

\newtheorem{proposition}[definition]{Proposition}

\newcommand{\R}{\mathbb{R}}

\newcommand{\na}{\nabla}

\newcommand{\p}{\partial}
\newcommand{\loc}{{\rm loc}}

\newcommand{\e}{\epsilon}

\newcommand{\dd}{{\rm d}}

\newcommand{\G}{\Gamma}
\newcommand{\linf}{{L^\infty}}

\newcommand{\two}{{\rm II}}

\newcommand{\A}{{\mathcal{A}}}

\newcommand{\I}{{\mathcal{I}}}

\def\XXint#1#2#3{{\setbox0=\hbox{$#1{#2#3}{\int}$ }
\vcenter{\hbox{$#2#3$ }}\kern-.6\wd0}}

\newcommand{\gauss}{{\kappa_{\Sigma}}}

\newcommand{\tg}{{\widetilde{\G}}}

\newcommand{\llinf}{L^\infty_t L^\infty_x}

\title{Existence of semiglobal $W^{2,p}$-isometric immersions for negatively curved surface metrics with unbounded second fundamental form}

\author{Siran Li}

\address{Siran Li: School of Mathematical Sciences $\&$ CMA-Shanghai, Shanghai Jiao Tong University, No.~6 Science Buildings,
800 Dongchuan Road, Minhang District, Shanghai, China (200240)}

\email{\texttt{siran.li@sjtu.edu.cn}}


\date{\today}

\begin{document}

\begin{abstract}
This paper is concerned with the existence theory of isometric immersions of surfaces with negative Gaussian curvature into the three-dimensional Euclidean space. We reformulate the Gauss--Codazzi equations, \emph{i.e.}, the partial differential equations for isometric immersions, into hyperbolic conservation laws for the flows of Chaplygin gas with nonzero source terms. Then, employing theories of invariant regions and compensated compactness,  we establish for any given $p \in [2,\infty[$ the existence of $W^{2,p}_\loc$-isometric immersions of various general families of metrics over infinite strips $\R\times [0,T]$ with arbitrarily large $T$. Such metrics include those of various classical minimal surfaces: helicoid, catenoid, and Enneper surfaces, as well as metrics in isothermal coordinates or of the ``reciprocal-type''. In our fluid dynamical formulation of the isometric immersion problem, we specialise in the case that the two Riemann invariants for the associated hyperbolic conservation law remain bounded and of distinctive signs, and obtain $L^p$-solutions to the initial-boundary value problem via the method of relative entropy with respect to a nonstationary ODE background. The 
isometric immersions constructed in this paper have second fundamental forms belonging to $L^p_\loc \setminus L^\infty_\loc$. 

\end{abstract}
\maketitle

%\tableofcontents

\section{Introduction}

\subsection{The problem}
The existence of isometric immersions of surfaces has been a long-standing problem in differential geometry, nonlinear Partial Differential Equations (PDE), and global analysis. The question, in its simplest form, reads as follows:
\begin{quotation}
    Given a positive definite $2 \times 2$-matrix field $g = \{g_{ij}\}$ on a 2-dimensional surface $\Sigma$, does there exist an immersion $\iota: \Omega \to \R^3$ such that $g = \iota^\#\delta$?
\end{quotation}
Throughout this paper, $\iota^\#$ denotes the pullback under $\iota$, and $\delta$ is the Euclidean metric on $\R^3$. We refer the reader to Janet~\cite{janet1926}, E. Cartan~\cite{cartan1927}, Aleksandrov~\cite{aleksandrov1948}, Nirenberg~\cite{nirenberg1953}, Nash~\cite{nash1954, nash1956}, Kuiper~\cite{kuiper1959},  Pogorelov~\cite{pogorelov1973}, Bryant--Griffiths--Yang~\cite{bgy83}, as well as the monograph~\cite{hanhong2006} by Han--Hong for cornerstone results on this question, among many other references.

An important line of research on the isometric immersions problem is via PDE analysis. The Gauss--Codazzi equations are the compatibility equations of curvatures for the existence of an isometric immersion $\iota: (\Sigma,g) \to (\R^3,\delta)$. The \emph{intrinsic geometry} of $\iota$ is determined solely by the metric $g$, which serve as the known coefficients of the Gauss--Codazzi equations, while the \emph{extrinsic geometry} is given by the second fundamental form $\two$, which serve as the unknowns of the PDE. Suppose that an isometric immersion $\iota$ exists; then, the flat Riemannian curvature on $(\R^3,\delta)$ projects either on two tangential directions of $\iota$, or on one tangential and one normal directions of $\iota$. Such projections give rise to one \emph{Gauss equation} and two \emph{Codazzi(--Mainardi) equations}, respectively:
\begin{equation}\label{eq:Gauss-Codazzi}
\begin{cases}
    LN-M^2 = \gauss,\\
    \p_x M - \p_y L = \G^2_{22} L - 2\G^2_{12} M + \G^2_{11}N,\\
    \p_xN-\p_yM = -\G^1_{22} L + 2\G^1_{12} M - \G^1_{11}N.
\end{cases}
\end{equation}
As is customary, we write $g = \begin{bmatrix}
    E&F\\
    F&G
\end{bmatrix}$ and  $\two=\begin{bmatrix}
    L&M\\
    M&N
\end{bmatrix}$. The Christoffel symbols $\left\{\G^i_{jk}\right\}_{i,j,k \in \{1,2\}}$ and Gaussian curvature $\gauss$ are determined by $g$. 

The Gauss--Codazzi system~\eqref{eq:Gauss-Codazzi} is also a sufficient condition for the \emph{local} existence of isometric immersions under mild regularity assumptions. This is known as the ``fundamental theorem of surface theory'', pioneered by Bonnet (1860) \cite{bonnet}. We refer to \S\ref{sec: fund theorem} for detailed discussions.

The sign of \(\gauss\), the Gaussian curvature of $g$, plays a decisive role in the study of the isometric immersion problem. The Gauss--Codazzi system~\eqref{eq:Gauss-Codazzi} is \emph{elliptic} where $\gauss>0$ and \emph{hyperbolic} where $\gauss<0$. There has been vast literature for the analysis of purely elliptic Gauss--Codazzi equations, \emph{i.e.}, when $\gauss>0$ everywhere on $\Sigma$. In this case, $\Sigma$ is homeomorphically the two-sphere ${\bf S}^2$ by the Gauss--Bonnet theorem, and the question of the existence of isometric immersions is known as the ``Weyl problem''. Nirenberg \cite{nirenberg1953} answered the Weyl problem in the affirmative for $C^4$-metrics; see also \cite{me-BLMS, guan1994, guanlu2017, pogorelov1973, iaia1985, buragoshefel, aleksandrov1948} for subsequent developments (including results concerning the degenerate-elliptic case $\gauss \geq 0$), among many other papers. 

In contrast, when $\gauss$ is strictly negative, the (global) existence of isometric immersions is known to have obstructions. A well-known theorem by Hilbert--Cohn-Vossen \cite{hilbert1999} states that the pseudosphere, \emph{i.e.}, the space form of constant Gaussian curvature \(-1\),  cannot be $C^2$-isometrically embedded into $(\mathbb{R}^{3},\delta)$. Efimov~\cite{efimov1963a, efimov1963b} extended this result by showing that complete negatively curved surfaces satisfying $\sup |\gauss|+ \sup \big|\nabla |\gauss|^{-1/2}\big| < \infty$ have no $C^3$-isometric immersion into $\R^3$; a particular case is when \(\gauss = -r^{-2}\) for in some geodesic polar coordinates \((\theta, r)\) as $r \to \infty$. Further work in this direction can be found in Poznyak~\cite{poznyak1973}, Rozendorn~\cite{rozendorn1992}, Tunitski~\cite{tunitski1993}, and Poznyak--Skikin~\cite{poznyakshikin1996}, etc. These developments led S.-T. Yau~\cite{yau2000} to propose the following problem: \emph{Find a sufficient condition for a complete negative curved surface to be isometrically embedded in \(\mathbb{R}^3\)}. Yau also suggested that reasonable conditions should involve decay rates of the Gaussian curvature at infinity. In this direction, Hong~\cite{hong1993} established the existence of global isometric immersions of surfaces such that $\partial_{r} \log (|K| r^{2+}) \leq 0$ for large $r$, and $\p_\theta\log|\gauss|$, $\p_{\theta\theta}\log|\gauss|$, and $r\p_\theta\p_r\log|\gauss|$ are bounded. Recently, in the nice work~\cite{chhw} by Cao--Han--Huang--Wang, Hong's pointwise conditions are relaxed to certain integral conditions, which allow $\gauss$ to exhibit non-monotone, even oscillatory behaviours as $r \to \infty$.

In passing, we comment that the existence theorems on solutions to the Gauss--Codazzi equations for metrics of sign-changing curvature are largely elusive in the literature. In this case, the Gauss--Codazzi equations are of mixed types, and the only available results are local in nature (Lin~\cite{lin1986}, Han~\cite{han2005}, and Chen--Clelland--Slemrod--Wang--Yang~\cite{chen2018}). Also, until recently, the existence theory for degenerate-hyperbolic Gauss--Codazzi equations has also been largely unknown. In this direction, Hu--Guo--Qin~\cite{hu1} and Hu--Guo--Li--Qin~\cite{hu2} obtained local existence theorems near the degenerate curve of $\gauss$ by analysing the fluid dynamical formulations.

\subsection{Hyperbolic PDE approach to isometric immersions}
\label{subsec: Hyperbolic PDE approach to isometric immersion}

In recent fifteen years or so, rapid developments have been witnessed in the research on the isometric immersion problem for negatively curved surfaces, by exploiting connections between the Gauss--Codazzi equations and various PDE models in compressible fluid dynamics. Chen--Slemrod--Wang \cite{chen2010a} first observed that Codazzi equations are essentially a system of hyperbolic balance laws for \(\{L, M\}\), once we substitute \(N=(\gauss + M^2)/L\) via the Gauss equation. In this formulation, if one suitably interprets $(L,M,N)$ in terms of fluid dynamical quantities, the cases  \(\gauss > 0\), \(\gauss = 0\), and \(\gauss < 0\) correspond respectively to the subsonic, sonic, and supersonic regimes for compressible flows.

In the supersonic case \(\gauss < 0\), the balance laws obtained by transformations applied to the Gauss--Codazzi system~\eqref{eq:Gauss-Codazzi} are strictly hyperbolic. By way of applying the methods of vanishing artificial viscosity and  invariant regions, thoroughly exploited in the works of Morawetz~\cite{morawetz1985, morawetz1995, morawetz2004} and DiPerna~\cite{diperna1983, diperna1985} on transonic flows, Chen--Slemrod--Wang \cite{chen2010a} established the  existence of global \(C^{1,1}\)-isometric immersions of a one-parameter family \(\{g^{(\beta)}\}_{\beta>1}\) of negatively curved metrics, where \(g^{(\sqrt{2})}\) corresponds to the metric of the standard catenoid. Later, using another closely related balance laws, Cao--Huang--Wang \cite{cao2015} proved the existence of global \(C^{1,1}\)-isometric immersions of another distinctive family of metrics with negative $\gauss$, which includes the standard helicoid as a special case.  Moreover, Cao--Huang--Wang \cite{cao2016} established the global existence of \(C^{1,1}\)-isometric immersions of metrics of the form \(g = E(y)\,dx^2 + dy^2\) via the Lax--Friedrich scheme, under the assumptions that \(\log(E(y)^2\sqrt{-\gauss})\) is a non-increasing \(C^{1,1}\)-function and \(E''(y) = -\gauss E(y)\). Christoforou~\cite{christoforou2012} established the existence of \(C^{1,1}\)-isometric immersions of conformal metrics \(g_{c,q^*} = (\cosh(cx))^{2/(q^{*2}-1)}(dx^2 + dy^2)\) with certain conditions on \(q^*\), where \(c > 0, q^* > 1\) are constants, by exploiting the theory of global BV-solutions to hyperbolic conservation laws (\textit{cf.} Lax~\cite{lax1954}, Dafermos~\cite{dafermos2016}, and Liu~\cite{liu2021}).\footnote{For the discussions in this paragraph on \cite{christoforou2012, cao2015, cao2016, chen2010a}, by global isometric immersions we mean those of an infinite strip $\R \times [a,b]$ into $\R^3$.} See also \cite{acharya2017, li, chen2010b, christoforou2023, li-survey, ls2023} for further developments and discussions on the fluid dynamical formulations of the Gauss--Codazzi equations. 

%To summarise, in the recent works \cite{cao2015, cao2016, chen2010a, chen2010c, christoforou2012}, various negatively curved metrics of the ``generalised catenoid type'' or the ``generalised helicoid type'' have been proved to admit global \(C^{1,1}\)-isometric immersions into \(\mathbb{R}^{3}\).

Our current paper aims to further studies on the isometric immersion problem for negatively curved surfaces through the analysis of the Gauss--Codazzi system~\eqref{eq:Gauss-Codazzi} in the hyperbolic regime. As is characteristic to hyperbolic PDE, even if one starts with smooth initial data, \emph{singularities may develop in finite time} (\emph{e.g.}, shock waves and vortex sheets). It is thus natural to investigate weak solutions in function spaces consisting of discontinuous functions. With this in mind, we identify $y$ with the time variable $t$ and look for $L^p$-solutions to~\eqref{eq:Gauss-Codazzi}, which in turn leads to $W^{2,p}$-isometric immersions via the fundamental theorem of surface theory with low regularity.

Apart from their intrinsic mathematical interest in PDE, analysis, and geometry, the existence, rigidity, and stability properties of Sobolev $W^{2,p}$- or fractional Sobolev $W^{s,p}$-isometric immersions are also of significance in the theory of nonlinear elasticity. We refer the reader to~\cite{ciarlet1988, ciarletgratie2008, ciarletlarsonneur2002, ciarletmardare2019,hornung2013, lewicka2011, lewicka2017, lewicka2020,kupferman2019,lps2024,alpern2022, alpern2024}, among many other references.

%Let us also remark that in our regularity class ($W^{2,p}$-isometric immersions for $p \in [2,\infty[$), isometric immersions are \emph{rigid}, in contrast to the ``wild'' or flexible \(C^{1,\alpha}\)-isometric immersions (with \(0 < \alpha \le \alpha_* < 1\) for some \(\alpha_*\)) constructed by convex integration~\cite{delellis2012}. 

%\blue{The main contribution of this paper is... ---[REWRITE]} to establish the global or ``nearly global'' existence of \(C^{1,1}\)—indeed, \(W^{2,\infty}\)—isometric immersions for several more families of such metrics. This is done by further exploiting the Gauss-Codazzi equations (1.4), (1.3) via the theory of hyperbolic balance laws, and by further exploring the structures of relevant geometric quantities, e.g., the Gauss curvature, the metric components, and the Christoffel symbols. In addition, more than simply producing isometric immersions, our approach in fact resolves an initial value problem of (a transformed version of) the Gauss-Codazzi equations.

\subsection{Fluid dynamical formulation for Gauss--Codazzi equations}\label{subsec: Chaplygin gas}

In line with Chen--Slemrod--Wang~\cite{chen2010a}, we further investigate the existence theory of isometric immersions of surface metrics with negative Gaussian curvature by exploring the \emph{fluid dynamical formulation} of the Gauss--Codazzi system~\eqref{eq:Gauss-Codazzi}. A key novelty is that we consider solutions to the Gauss--Codazzi equations that are unbounded (in the ``initial layer''): our analysis encompasses the case that $\two=\begin{bmatrix}
    L&M\\
    M&N
\end{bmatrix}$ lies in $L^p_\loc(\R \times [0,T])$ for $2\leq p<\infty$ but not in $L^\infty_\loc(\R\times [0,\delta_0])$ for any $\delta_0>0$. 

As in \cite{chen2010a, cao2015, cao2016, li}, the starting point of our analysis is the introduction of the fluid variables, ``density'' and ``momentum'':
\begin{equation}\label{rho, m, def}
    \rho := \frac{L}{\gamma},\qquad
     m = -\frac{M}{\gamma}.
\end{equation}
Then, by expressing $N$ via the Gauss equation $LN-M^2 = \gauss$, one arrives at the system of balance laws:
\begin{equation}\label{rho, m equation}
    \begin{cases}
        \p_t \rho + \p_x m = -\tg^{2}_{22}\rho - 2\tg^{2}_{12}m - \tg^2_{11} \left(\frac{m^2-1}{\rho}\right),\\
        \p_t m + \p_x\left(\frac{m^2-1}{\rho}\right) = -\tg^1_{22}\rho -2\tg^1_{12}m - \tg^1_{11}\left(\frac{m^2-1}{\rho}\right).
    \end{cases}
\end{equation}
Here and hereafter, we shall designate
\begin{equation}\label{gamma}
\gamma := \sqrt{-\gauss},    
\end{equation}
which is well-defined for $\gauss<0$, and introduce the modified Christoffel symbols:
\begin{equation}\label{modified Gamma}
    \begin{cases}
\tg^1_{22} := \G^1_{22},\qquad\tg^1_{12} := \G^1_{12} + \frac{\p_y\gamma}{2\gamma},\qquad\tg^1_{11} := \G^1_{11} + \frac{\p_x \gamma}{\gamma},\\
\tg^2_{22} := \G^2_{22} + \frac{\p_y\gamma}{\gamma},\qquad\tg^2_{12} := \G^2_{12} + \frac{\p_x\gamma}{2\gamma},\qquad\tg^2_{11} := \G^2_{11}.
    \end{cases}
\end{equation}
Let us also denote
\begin{align*}
 n :=   \frac{m^2-1}{\rho} = \frac{N}{\gamma}.
\end{align*}

Throughout this paper, we seek solutions such that 
\begin{equation}\label{rho geq 0}
  \rho \geq 0, 
\end{equation}
or equivalently $L \geq 0$. The case $L \leq 0$ shall not be considered separately in this paper, since the Gauss--Codazzi system~\eqref{eq:Gauss-Codazzi} is invariant under $(x,y; L,M,N) \mapsto (x,y; -L,-M,-N)$. Here, let us also point out a flawed statement in \cite[the end of Step~1 on p.426]{li}:\footnote{The author thanks Prof.~Raz Kupferman for pointing this out in personal communication.} the author claimed that ``As $g$ is symmetric and positive definite, [$\rho =0 \Leftrightarrow L=0$] is automatically satisfied''. Indeed, $(L,M,N)=(0,\gamma,0)$ is a solution to~\eqref{eq:Gauss-Codazzi} if $\na \log \gamma=-2\left[\G^2_{12}, \G^1_{12}\right]^\top$, which equals $-\left[\p_y\log G, \p_x\log E\right]^\top$ when $F \equiv 0$. In view of Brioschi's formula~\eqref{Brioschi}, this condition is satisfied for metrics $g$ in which $F\equiv 0$ and $\big(E=E(x,y), G=G(x,y)\big)$ satisfies the third-order PDE:
\begin{equation*}
    \begin{cases}
        \p_x \log \left\{\frac{1}{\sqrt{EG}}
\left[
\frac{\partial}{\partial x}\left(\frac{\p_x G}{\sqrt{EG}}\right)
+
\frac{\partial}{\partial y}\left(\frac{\p_y E}{\sqrt{EG}}\right)
\right]\right\} = -4 \p_y\log G,\\
\p_y\log \left\{\frac{1}{\sqrt{EG}}
\left[
\frac{\partial}{\partial x}\left(\frac{\p_x G}{\sqrt{EG}}\right)
+
\frac{\partial}{\partial y}\left(\frac{\p_y E}{\sqrt{EG}}\right)
\right]\right\} = 4\p_x\log E.
    \end{cases}
\end{equation*}

On the other hand, one crucial novelty of this work is that we do \emph{not} exclude concentration formation; that is, we allow degeneracy $\rho = +\infty$, or equivalently $L = +\infty$ at some point $P$. For reasons that will soon become transparent, we also impose $|M| \leq \gamma$ everywhere. Thus, $L=+\infty$ forces $N=0$ at $P$. But the Gauss--Codazzi Eq.~\eqref{eq:Gauss-Codazzi} is invariant under $(x,y;L,M,N) \mapsto (x,y;-L,-M,-N)$ and $(x,y;L,M,N) \mapsto (y,x;N,M,L)$, so one does not have to consider separately the case $\rho = 0$, \emph{i.e.}, the formation of vacuum.

For subsequent analysis we shall focus on~\eqref{rho, m equation}, a 1D isentropic Euler equation with nonzero source terms. Once we write $$m=\rho u,$$ the homogeneous equation associated with~\eqref{rho, m equation} becomes
\begin{equation}\label{homogeneous eq, chaplygin}
\begin{cases}
\p_t\rho + \p_x(\rho u) = 0,\\
\p_t(\rho u) + \p_x \left(\rho u^2 + P\right) = 0,
    \end{cases}
\end{equation}
with 
\begin{align*}
    P = P(\rho) = -\frac{1}{\rho}.
\end{align*}
This is the pressure law for the Chaplygin gas, as obtained in Chen--Slemrod--Wang~\cite{chen2010a}.

%In an earlier nice work, Cao--Huang--Wang \cite[Theorem~5.1]{cao2016} proved the global existence of $\linf$-solutions to Eq.~\eqref{eq:Gauss-Codazzi} for  the negatively curved metric of the special form $g = E(y) dx^2 + dy^2$ with $\p_{yy}E = -\gauss E$, $E|_{\{y=0\}}=1$, and $\p_yE|_{\{y=0\}}=0$, subject to certain  conditions on $\gamma=\gamma(y)$, by way of  applying to Eq.~\eqref{rho, m equation} the Lax--Friedrichs finite-difference scheme with  fractional step. Such $g$ is said to be a helicoid-type metric, as $E(y)=y^2+c^2$ for some $c>0$ yields the standard parametrisation for the helicoid. In contrast, we shall establish in this paper the  existence global or ``semiglobal'' $L^p$-solutions to Eq.~\eqref{eq:Gauss-Codazzi} for more general families of metrics. In passing, we note that Eq.~\eqref{rho, m equation} echoes the fluid formulation in \cite{chen2010a}, which recasts the Gauss--Codazzi equations into a 2D inhomogeneous compressible Euler system for Chaplygin gas. 

Regarding the conservation law~\eqref{homogeneous eq, chaplygin} for the 1D Chaplygin gas, direct computation shows that two families of \emph{contact discontinuities} propagate at speeds
\begin{equation}\label{eigenvalues}
    \lambda_\pm = \frac{m \pm 1}{\rho}.
\end{equation}
In particular, there is no onset of shock waves. Hence, it is natural to expect that global weak solutions to the homogeneous Eq.~\eqref{homogeneous eq, chaplygin} may exist if, for some finite number $C_0$, 
\begin{equation}\label{lambda pm condition}
-C_0  \leq \lambda_- \leq 0 \leq \lambda_+ \leq C_0  
\end{equation} 
holds for all time. In particular, the sign conditions $\lambda_-\leq 0 \leq \lambda_+$ imply that $$-1 \leq m \leq 1.$$ By definition of the fluid variables $(\rho, u)$ in~\eqref{rho, m, def} and the Gauss equation $LN-M^2 = \gauss$, we have $LN \leq 0$ and hence $$N \leq 0.$$ Also note that by~\eqref{eigenvalues} and \eqref{lambda pm condition} we have $\frac{1\pm m}{\rho} \leq C_0$, thus $$\rho \geq \frac{1}{C_0}>0.$$ In particular, there is no vacuum formation.

The case $\rho=\infty$ suggests that the Gauss--Codazzi system~\eqref{eq:Gauss-Codazzi} may admit non-$\linf$-weak solutions. Such solutions evade the earlier constructions in \cite{cao2015, cao2016, chen2010a, chen2010c, li}. A key novel feature of our current paper is to establish an $L^p$-solution theory $(p<\infty)$ via \emph{entropy analysis} based on the choice of singular entropies for the balance law~\eqref{rho, m equation} of the Chaplygin gas.

\subsection{Analytic goal} To sum up the preceding discussions, our aim from the PDE perspectives is as follows:
\begin{tcolorbox}[
  boxrule=0.5pt,
  colback=white,
  colframe=blue,
  left=1em,
  right=1em,
]
\begin{quotation}
Solving the Gauss--Codazzi system~\eqref{eq:Gauss-Codazzi} subject to constraints:
\begin{equation}\label{L,M,N sign constraints}
    \begin{cases}
    L \geq 0, \quad N \leq 0, \quad\text{and } -\gamma \leq M \leq \gamma \quad\text{a.e.},\\
    \frac{M \pm \gamma}{L} \in \linf,
    \end{cases}
\end{equation}
while allowing $L=\infty$, $N=0$, and $M = \pm\gamma$ somewhere in the domain.
\end{quotation}
\end{tcolorbox}

This goal shall be achieved through a combination of the method of invariant regions (\emph{i.e.}, a maximum principle for parabolic systems of PDE; see Chueh--Conley--Smoller~\cite{chueh1977} and Smoller~\cite{smoller1994}), the theory of compensated compactness (for passing to the vanishing viscosity limits; see Dafermos~\cite{dafermos2016} and Evans~\cite{evans1990}), and some delicate transforms among the fluid variables $(\rho, m)$, the Riemann invariants $(w=\lambda_+,z=\lambda_-)$, as well as the geometric variables $(L, M,N)$.

\subsection{Main results}

Recall from \S\ref{subsec: Hyperbolic PDE approach to isometric immersion} that the existence of global $W^{2,\infty}$-isometric immersions of three families of negatively curved metrics has been established, as summarised in Table~\ref{table: 1}:

\begin{table}[H]
\centering
\caption{Families of negatively curved metrics admitting $C^{1,1}$-isometric immersions}
\begin{tabular}{
    >{\centering\arraybackslash}p{5cm}
    >{\centering\arraybackslash}p{2.8cm}
    >{\centering\arraybackslash}p{3.2cm}
    >{\centering\arraybackslash}p{4.2cm}
}
\toprule
\textbf{Metric} & \textbf{Shape} & \textbf{Domain} & \textbf{References} \\
\midrule
\begin{small}$g = (\cosh(cx))^{\frac{2}{(q^{*2}-1)}}(dx^2 + dy^2)$ with $c>0, q^*>1$
\end{small}
& generalised catenoids 
& $[-x_0, x_0] \times \mathbb{R}$ 
& Chen--Slemrod--Wang \cite{chen2010a}; Christoforou \cite{christoforou2012} \\
\midrule
\begin{small}$g = E(y)\,dx^2 + dy^2$ with $E'' = -\gauss E$ \end{small}
& generalised helicoids 
& $\mathbb{R} \times [0, y_0]$ 
& Cao--Huang--Wang \cite{cao2015, cao2016}; Li \cite{li} \\
\midrule
\begin{small}
$g = (1+x^2+y^2)^{\alpha}\,(dx^2+dy^2)$ with $\alpha \in [1,10]$ 
\end{small}
& generalised Enneper  
& $\mathbb{R}^2 \setminus \{x = \pm y\}$ 
& Li \cite{li} \\
\bottomrule
\end{tabular}\label{table: 1}
\end{table}

In this paper, we substantially augment Table~\ref{table: 1} by showing that fifteen families of metrics with negative Gaussian curvature admit global $W^{2,p}$-isometric immersions   for any given $2 \leq p < \infty$. We say that an isometric immersion is \emph{semiglobal} if it is defined over arbitrarily large strips; \emph{i.e.},  $\R\times[0,T]$ with arbitrarily large (but finite)  $T>0$.

Our main result, in a simplified formulation, is as follows:

\begin{theorem}\label{thm: main}
For any $p \in [2,\infty[$, the fifteen negatively curved metrics families in Table~\ref{table: 2} admit semiglobal isometric immersions in $(\R^3, \delta)$ of regularity $W^{2,p}_\loc$. Moreover, we can require their second fundamental forms to lie in $L^p_\loc \setminus L^\infty_\loc$. 
\end{theorem}

\begingroup
\small
\setlength{\tabcolsep}{3pt}
\renewcommand{\arraystretch}{1.15}

\begin{longtable}{
    @{}
    >{\centering\arraybackslash}p{7cm}
    >{\centering\arraybackslash}p{2.8cm}
    >{\centering\arraybackslash}p{4.5cm}
    @{}
}

\caption{Families of negatively curved metrics admitting
\(W^{2,p}_{\loc}\)-isometric immersions. Throughout, \(\delta_0>0\) is an arbitrarily
small constant and \(T>0\) is an arbitrarily large constant.}
\label{table:2}
\\

\toprule
\textbf{Metric}
&
\textbf{Shape}
&
\textbf{Domain}
\\
\midrule
\endfirsthead

% Header repeated at the beginning of every subsequent page
\multicolumn{3}{c}{%
\tablename~\thetable\ (continued)
}
\\
\toprule
\textbf{Metric}
&
\textbf{Shape}
&
\textbf{Domain}
\\
\midrule
\endhead

% Footer printed at the bottom of every non-final page
\midrule
\multicolumn{3}{r}{\emph{Continued on the next page}}
\\
\endfoot

% Footer printed only at the end of the table
\bottomrule
\endlastfoot

\(g=(ay^2+by+c)\,dx^2+dy^2\), where
\(a>0\) and \(b^2-4ac<0\)
&
helicoid-type
&
\(\R\times[-b/(2a),T]\)
\\
\midrule

\(g=e^{ay}\,dx^2+dy^2\), where \(a>0\)
&
helicoid-type
&
\(\R\times[0,T]\)
\\
\midrule

\(g=y^q\,dx^2+dy^2\), where \(q>2\)
&
helicoid-type
&
\(\R\times[\delta_0,T]\)
\\
\midrule

\begin{small}
$g =\cosh(ay)\,dx^2 + dy^2$ with $a >0$\end{small} 
& helicoid-type
& $\R\times [\delta_0,T]$
\\
\midrule
\begin{small}
$g =(a-y)^{-b}\,dx^2 + dy^2$ with $a,b>0$ \end{small} 
& helicoid-type
& $\R\times[-T,a-\delta_0]$ 
\\
\midrule
\begin{small}
$g =(\sinh y+c)\,dx^2 + dy^2$ with $c>1$\end{small} 
& helicoid-type
& \begin{footnotesize}$\R\times \left[{\rm arcsinh}\left(\sqrt{2c^2+2}-c\right)+\delta_0,T\right]$\end{footnotesize} 
\\
\midrule
\begin{small}
$g = \cosh^a y\,dx^2 + dy^2$ with $a \geq 2/3$\end{small} 
& helicoid-type
& \begin{footnotesize}$\R\times ]0,T[$ \end{footnotesize}
\\
\midrule
\begin{small}
$g = y^{2}(1+y^2)^{2\beta}\,dx^2 + (1+y^2)^{2\beta}\,dy^2$ with $\beta > 0$\end{small} 
& generalised Enneper surfaces
& $\R \times [\delta_0,T]$ 
\\
\midrule
\begin{small}
$g = e^{ay}\,dx^2 + e^{-ay}\,dy^2$ with $a>0$\end{small} 
& reciprocal-type
& $\R \times [\delta_0,T]$
\\
\midrule
\begin{small}
$g=A\cosh (\omega y) \, dx^2 + \left(A\cosh (\omega y)\right)^{-1}\,dy^2$ with $A, \omega>0$\end{small} 
& reciprocal-type
& $\R \times [0,T]$
\\
\midrule

\begin{small}
$g=(1+y^2)\,dx^2+\frac{dy^2}{1+y^2}$\end{small} 
& reciprocal-type
& $\R \times [0,T]$
\\
\midrule

\begin{small}
$g=e^{ay^q}\left(dx^2+dy^2\right)$ with $a>0$, $1 <q \leq 2$\end{small} 
& conformal
& $\R \times \left[\left(\frac{2-q}{aq}\right)^{\frac 1q},T\right]$
\\
\midrule

\begin{small}
$g=e^{ay^2+by}\left(dx^2+dy^2\right)$ with $a,b>0$\end{small} 
& conformal
& $\R \times [\delta_0,T]$ 
\\
\midrule
\begin{small}
$g=e^{a \sinh y}\left(dx^2+dy^2\right)$ with $a>0$\end{small} 
& conformal
& \begin{footnotesize}$\R \times \left[\delta_0 ,T\right]$\end{footnotesize}
\\
\midrule
\begin{small}
$g=\cosh^a (y)\left(dx^2+dy^2\right)$ with $a\geq 2$\end{small} 
& conformal
& $\R \times [\delta_0,T]$ 

% Insert the remaining rows here, using exactly the same format.
% Do not place \bottomrule manually after the final row, because
% it is already supplied by \endlastfoot.

\label{table: 2}
\end{longtable}

\endgroup

In fact, we shall establish the existence results of isometric immersions for negatively curved metrics under more general structural conditions  easily checkable for the examples in Table~\ref{table: 2}. 
\begin{itemize}
    \item 
    For metrics of the form $$g=E(x,y)dx^2 + G(y)dy^2,$$ we specify the sign conditions on certain modified Christoffel symbols $\tg^i_{jk}$ (see \eqref{modified Gamma}) that warrant the existence of invariant regions for Riemann invariants. This leads to $W^{2,p}$-isometric immersions via compensated compactness and entropy analysis over arbitrarily large (but finite) rectangles $[a,b]\times[0,T]$.
    
    \item 
    If we further restrict ourselves to metrics of the form $$g=E(y)dx^2 + G(y)dy^2,$$ we obtain semiglobal $W^{2,p}_\loc$-isometric immersions over $\R\times[0,T]$ with arbitrarily large (but finite) $T>0$ by employing the relative entropy Theorem~\ref{thm:global-relative-entropy}, which may be of independent interest beyond the isometric immersion problem. 
    
\end{itemize}

\begin{theorem}\label{thm: more general metrics}
\begin{enumerate}
    \item 
    Consider $g = E(x,y)\,dx^2 + G(y)\,dy^2$ with $E, G>0$ of finite $W^{3,\infty}$-norm on $\I\times[0,T]$, where $\gamma = \sqrt{-\gauss}>0$, $T$ is any positive number, and $\I\subset\R$ is any bounded interval. Suppose that $$\p_y\left(E\gamma\right) \geq 0\qquad\text{a.e. on $\I\times[0,T]$}.$$ Then, for any given $p \in [2,\infty[$, there exists an isometric immersion $\iota:(\I\times[0,T], g) \to (\R^3,\delta)$ of $W^{2,p}$-regularity. Moreover, we can require their second fundamental forms to lie in $L^p \setminus L^\infty$. 

\smallskip

\item 
If, in addition, $E(x,y)\equiv E(y)$ and $\log(\gamma) \in W^{1,\infty}(\R\times [0,T])$, then there exists a semiglobal isometric immersion $\iota:(\R\times[0,T], g) \to (\R^3,\delta)$ of $W^{2,p}_\loc$-regularity. Moreover, we can require their second fundamental forms to lie in $L^p_\loc \setminus L^\infty_\loc$. 
\end{enumerate}

\end{theorem}

Geometrically, metrics of the form $g=E(y)dx^2+G(y)dy^2$ (in particular, those in Theorem~\ref{thm: more general metrics} (2) and Table~\ref{table: 2}) all have $\p_x$ as a Killing vector field, so these are cohomogeneity-one metrics.  Introduce the geodesic variable $
ds=\sqrt{G(y)}\,dy$
and define $
B(s):=\sqrt{E(y(s))}$. The metric (with primes denoting differentiation with respect to \(s\))  can be written as
\begin{equation}
g=B(s)^2\,dx^2+ds^2.
\label{eq:metric-geodesic-coordinates}
\end{equation}
The Gaussian curvature is $
\gauss=-\frac{B''}{B}$, and the conditions in Theorem~\ref{thm: more general metrics} are then equivalent to 
\begin{equation*}
B'' > 0
\qquad\text{and}\qquad 
(B^2\gamma)'\geq0 \Longleftrightarrow -\frac{\gamma'}{\gamma}
    \leq
    2\frac{B'}{B}.
\end{equation*}
This describes a substantially more general class of negatively
curved metrics than the particular explicit families listed in Theorem~\ref{thm: main},  Table~\ref{table: 2}. Note that Theorem~\ref{thm: more general metrics} (2) also applies to metrics of the form
\begin{align*}
    g = E(y)dx^2+2F(y)dxdy + G(y)dy^2,
\end{align*}
as we can eliminate $2F(y)dxdy$ by taking the coordinate transform $$x' = x+\int^y \frac{F(t)}{E(t)}\,\dd t$$ and thus rewriting $$g=E(y)dx'^2 + \left(G(y)-\frac{F(y)^2}{E(y)}\right)dy^2.$$

Our restriction in Theorem~\ref{thm: more general metrics} (2) to $E(x,y)\equiv E(y)$ for semiglobal isometric immersion is due to  a technical reason: we employ entropy estimates to obtain uniform $L^p_\loc$-bounds for the approximate solutions to the Gauss--Codazzi equations (or their fluid formulation in the $(\rho,m)$ variables). This  works only on bounded domains; see Remark~\ref{remark: finite domain} in \S\ref{sec: entropy}. To obtain semiglobal immersions, we need to modify the entropy arguments by considering the \emph{relative entropy}  \emph{\`{a} la} C. Dafermos with respect to a background state lying inside the invariant regions. The nonlinear feature of the Gauss--Codazzi equations forbids the existence of a constant equilibrium state; instead, we resort to comparison arguments to construct  background states from ODEs derived for $g$ with all metric coefficients depending only on one variable. This motivates the following: 
\begin{theorem}
\label{thm:global-relative-entropy}
Let
$U^\e = [\rho^\e, m^\e]^\top$ be a smooth solution in $[0,T]\times\mathbb{R}$ to the balance law
\begin{equation*}
\partial_tU^\varepsilon
+
\partial_xF(U^\varepsilon)
=
\varepsilon\partial_{xx}U^\varepsilon
+
S(U^\varepsilon),
\end{equation*}
where $F(\rho,m)
=\begin{bmatrix}
    m\\
    \dfrac{m^2-1}{\rho}
\end{bmatrix}$ and $S(\rho,m)
=
\begin{bmatrix}
-a(t)\rho-2b(t)m-c(t)\left(\dfrac{m^2-1}{\rho}\right)\\
-2k(t)m
\end{bmatrix}$. Assume that $a,b,c,k\in L^\infty(]0,T[)$, $k(t)\geq 0$ for \emph{a.e.} $t\in[0,T]$, and $\rho^\e\geq\rho_*>0$, $-1<m^\e<1$ on $[0,T]\times\mathbb{R}$ for $\rho_*$ independent of $\e$. Set $\overline U(t):=[R(t),0]^\top$, where $R$ is the solution to the ODE: 
\begin{equation*}
R'(t)
=
-a(t)R(t)+\frac{c(t)}{R(t)}
\end{equation*}
that satisfies $0<R_*\leq R(t)\leq R^*<\infty$ for all $t \in [0,T]$. Then for given $p\in[2,\infty[$ we have  
\begin{equation*}
\sup_{\e>0}
\sup_{0\leq t\leq T}
\int_{\mathbb{R}}
H_p\bigl(
U^\e(t,x)\mid\overline U(t)
\bigr)\,\dd x
<\infty,
\end{equation*}
provided that $\sup_{\e>0}
\int_{\mathbb{R}}
H_p\bigl(
U^\e_0(x)\mid\overline U(0)
\bigr)\,\dd x
<\infty$. In particular, 
\begin{equation*}
\sup_{\varepsilon>0}
\left\|
\rho^\varepsilon-R
\right\|_{L^\infty(0,T;L^p(\mathbb{R}))}
<\infty.
\end{equation*}
Moreover, we have the gradient estimate:
\begin{align*}
\sup_{\e>0}\left\{\e\iint_{[0,T]\times\R} |\p_xU^\e|^2 \,\dd t\,\dd x\right\} < \infty.
\end{align*}
Here the relative entropy is defined as follows:
\begin{align*}
&H_p\left(U=[\rho,m]^\top\mid\overline U=[R,0]^\top\right)
\\
&\quad:= 
\rho^p \left[(1+m)^{1-p} + (1-m)^{1-p}\right]
-
2R^p
-
2pR^{p-1}(\rho-R).
\end{align*}
\end{theorem}

\begin{remark}\label{remark: special isom imm}
Along the proof of Theorem~\ref{thm: more general metrics}, by a straightforward application of Theorem~\ref{thm:global-relative-entropy} to the parabolically regularised version of~\eqref{rho, m equation} with artificial viscosity, we obtain the following: 
\begin{quote}
    Given a metric $g=E(y)dx^2+G(y)dy^2$ on $\R\times[0,T]$ with $\frac{d}{dy}(E\gamma) \geq 0$ a.e., $E, G>0$ belong to $W^{3,\infty}$, and $\log\gamma \in W^{1,\infty}$. For arbitrary $T>0$, there exists a $W^{3,\infty}$-isometric immersion of $(\R \times [0,T],g)$ into $(\R^3,\delta)$.
\end{quote}
\end{remark}

\subsection{Outline of the proof}
Our proof of the Main Theorems~\ref{thm: main}, \ref{thm: more general metrics} is outlined as follows:

\begin{enumerate}
    \item 
First, consider the fluid dynamical formulation for the Gauss--Codazzi system~\eqref{eq:Gauss-Codazzi}; namely, the balance law~\eqref{rho, m equation} for 1D Chaplygin gas. By adding artificial viscous terms $\left[    \e\p_{xx}\rho,\,
    \e\p_{xx}m
\right]^\top,$  we introduce the \emph{parabolic regularisation}~\eqref{parabolic regularisation} of the balance law.

\item
Next, we introduce the \emph{Riemann invariant coordinates} $\lambda_\pm$ as in~\eqref{eigenvalues} for the parabolically regularised system.  By the method of \emph{invariant regions}, the $\linf$-bounds for $\lambda_\pm$ can be established uniformly in $\e$ subject to ready-to-check structural conditions on the PDE~\eqref{rho, m equation}, formulated in terms of the signs of Christoffel symbols and $\na\log\gamma$, where $\gamma=\sqrt{-\gauss}$.

\item 
Then, by an entropy analysis involving singular entropies that blow up as the Riemann invariants degenerate to zero, we derive uniform $L^p$-bounds for the fluid variables $(\rho, m)$ --- or equivalently, for the geometric variables 
 $\two=\begin{bmatrix}
    L&M\\
    M&N
\end{bmatrix}$  --- in the presence of concentration $(\rho=\infty)$ for any $p \in [2,\infty[$ over arbitrarily large but finite rectangular domains. In this process, the uniform bounds for the $\linf$-norm of Riemann invariants obtained in Step~(2) will be crucially exploited.\footnote{Here, the variables $\rho, m, L,M,N\ldots$ depend on the regularisation parameter $\e$. In later parts, we write a superscript ${}^\e$ to emphasise this dependency. Any useful bound in this paper is uniform in $\e$.}  

\item 
In the case of special metrics $g=E(y)dx^2+G(y)dy^2$, we may upgrade the analysis in the above steps to the strip $\R\times[0,T]$, thanks to the technical Theorem~\ref{thm:global-relative-entropy} via relative entropy arguments. By carefully analysing the choice of initial data, we argue that the width $T>0$ of the strip can be chosen arbitrarily large.

\item 
With the uniform $L^p_\loc$-bounds for $(L,M,N)$ at hand, we may pass to the weak limits to obtain weak solutions to the Gauss--Codazzi system~\eqref{eq:Gauss-Codazzi}, via a \emph{compensated compactness} framework 
(\emph{cf.} \cite{lisu, chen2010a, chen2010c, chenli2018}) established for $L^p_\loc$-solutions. See Lemma~\ref{lem: comp comp framework}.

\item 
Once the weak solutions to the Gauss--Codazzi equations in $L^p$ are obtained, $W^{2,p}_\loc$-isometric immersions are constructed by using the \emph{fundamental theorem of surface theory} of low regularity, as detailed in \S\ref{sec: fund theorem}. 

%\item 
%Various families of negatively curved surface metrics including classical minimal surfaces (\emph{e.g.}, helicoid, catenoid, and Enneper surfaces) as special cases, satisfy the structural conditions in Step~(2) above. Then, these families of metrics admit invariant regions for the Riemann invariants $(w^\e, z^\e)$, and hence admitting semiglobal $W^{2,p}_\loc$-isometric immersions for any given $p \in [2,\infty[$. 

\end{enumerate}

%We emphasise once again that the existence of isometric immersions proved in this paper are semiglobal, \emph{i.e.}, proved over finite but arbitrarily large rectangular domains $\I \times [0,T]$ in $\R^2$. Indeed, the invariant region argument in Step~(2) of the outline above requires the density to be bounded from below away from $0$, while the entropy analysis imposes integrability conditions for $\rho$ on $\I$. These altogether force $\I$ to be bounded. Meanwhile, the entropy analysis is based on an argument involving Gr\"{o}nwall's inequality without general decaying mechanisms; see the proof of Proposition~\ref{propn: entropy analysis} (2) below. This requires $T$ to be finite.

\subsection{Organisation}
The remaining parts of the paper are organised as follows:

\begin{itemize}
    \item 
    In \S\ref{sec: prelims} we present some background knowledge of isometric immersions, Gauss--Codazzi equations, and differential geometry of surfaces. 
    
    \item 
    In \S\ref{sec: inv} we introduce the parabolically regularised PDE systems for the fluid variables $(\rho, m)$ and the corresponding Riemann invariants. Moreover, we also present an easy-to-check sufficient condition for the existence of invariant regions (Lemma~\ref{lem: Riem inv}). 
    
    \item 
    In \S\ref{sec: entropy}, we obtain uniform bounds on the density $\rho$ and the first derivatives $\sqrt{\e}\p_x \rho$ and $\sqrt{\e}\p_x m$ over arbitrarily large but finite domains via entropy and energy methods, independently of the parabolic regularisation 
    parameter $\e$.

    \item 
    In \S\ref{sec: new}, we prove Theorem~\ref{thm:global-relative-entropy} via relative entropy arguments, thus establishing uniform  estimates for $(\rho, m)$ on $\R \times [0,T]$.

    \item 
    In \S\ref{sec: existence}, using previously uniform estimates for fluid variables and employing the theory of compensated compactness and the fundamental theorem of surface theory in the weak regularity regime, we prove the existence of semiglobal $W^{2,p}_\loc$-isometric immersions for negatively curved surface metrics whenever the invariant regions exist.

    \item 
    In \S\ref{sec: conclusion}, we conclude the proof of Theorems~\ref{thm: main} and \ref{thm: more general metrics}.

\end{itemize}

\section{Preliminaries}\label{sec: prelims}

In this section, we present some background knowledge on differential geometry of surfaces. See, \emph{e.g.}, \cite{docarmo1992, eisenhart1997, hanhong2006} for more comprehensive treatments.

\subsection{Intrinsic geometry of surfaces}

Let $\Sigma$ be a surface, namely a 2-dimensional differentiable manifold, equipped with a Riemannian metric $g$; as usual we write $g_{11}=E$, $g_{12}=g_{21}=F$, and $g_{22}=G$. The Levi-Civita connection of $g$ is described by the Christoffel symbols $\left\{\G^i_{jk}\right\}_{i,j,k\in\{1,2\}}$:
\begin{align*}
    \G^i_{jk} = \frac{1}{2} \sum_{i,j,k,l=1}^2 g^{i\ell}\left(\p_j g_{kl} + \p_{k}g_{\ell j} - \p_\ell g_{jk}\right),
\end{align*}
where $g^{-1} = \{g^{jk}\}$. By Gauss's Theorema Egregium, the Gaussian curvature is determined solely by $g$. More explicitly, one has \emph{Brioschi's formula} in any local chart with coordinate $(x,y)$:
\begin{align}\label{Brioschi}
\gauss = \frac{
\left[
\begin{vmatrix}
-\frac{1}{2}\partial_{yy}E + \partial_{xy}F - \frac{1}{2}\partial_{xx}G & \frac{1}{2}\partial_x E & \partial_x F - \frac{1}{2}\partial_y E \\
\partial_y F - \frac{1}{2}\partial_x G & E & F \\
\frac{1}{2}\partial_y G & F & G
\end{vmatrix}
-
\begin{vmatrix}
0 & \frac{1}{2}\partial_y E & \frac{1}{2}\partial_x G \\
\frac{1}{2}\partial_y E & E & F \\
\frac{1}{2}\partial_x G & F & G
\end{vmatrix}
\right]}{(EG - F^2)^2}.
\end{align}
When the metric is in the diagonal form,\footnote{Note that $g$ is a symmetric matrix field in any local coordinate system, so it can always be diagonalised therein.} namely $F\equiv 0$, it reduces to
\begin{align*}
\gauss = -\frac{1}{2\sqrt{EG}}
\left[
\frac{\partial}{\partial x}\left(\frac{\p_x G}{\sqrt{EG}}\right)
+
\frac{\partial}{\partial y}\left(\frac{\p_y E}{\sqrt{EG}}\right)
\right].
\end{align*}

\subsection{Isometric immersions and Gauss--Codazzi equations} We consider the isometric immersion $\iota:(\Sigma,g)\to (\mathbb{R}^{3},\delta)$, the Euclidean 3-space. If such $\iota$  exists, then for each point $x\in \Sigma$ the tangent space $T_{\iota(x)}\mathbb{R}^{3}$ splits orthogonally: 
\begin{equation}\label{split}
T_{\iota(x)}\mathbb{R}^{3}\cong T_{x}\Sigma\bigoplus [T_{x}\Sigma]^{\perp}.    
\end{equation}
Denoting by $\overline{\nabla}$ the Levi-Civita connection on $T\mathbb{R}^{3}$ and by $\nabla$ the Levi-Civita connection on $T\Sigma$, let us define $\two' :\Gamma (T\Sigma)\times \Gamma (T\Sigma)\to \Gamma (T\Sigma^{\perp})$ via
\[
\two' (X,Y) := \overline{\nabla}_{X}Y - \nabla_{X}Y, 
\]
where $\G(T\Sigma)$ and $\G(T\Sigma^\perp)$ are the spaces of tangential and normal vector fields along the immersed image of $\Sigma$ via $\iota$. The tensor field $\two'$ describes the \emph{extrinsic geometry} of the isometric immersion $\iota$, \emph{i.e.}, the manner in which $\Sigma$ is immersed in the ambient space $\R^3$. We shall write $\two'=h_{11}\,dx^2 + 2h_{12}\,dxdy + h_{22}\,dy^2$ and 
\begin{align*}
\two'(X,Y,\nu) \equiv \delta\left(\two'(X,Y),\nu\right)\quad\text{for any $X,Y \in \G(T\Sigma)$ and $\nu \in \G(T\Sigma^\perp)$.}
\end{align*} 

The Gauss and Codazzi equations express the splitting of the zero Riemann curvature of $(\mathbb{R}^{3}, \delta)$ along~\eqref{split}. Let $X, Y, Z, W \in \Gamma(T\mathcal{M})$ be tangential vector fields and let $\nu \in \Gamma(T\mathcal{M}^{\perp})$ be a normal vector field. Denote the inner products induced by both $g$ or $\delta$ by $\langle \cdot , \cdot \rangle$, which makes sense for isometric immersion $\iota$. The Gauss equation is then expressed as
\begin{equation}\label{gauss, original}
R(X,Y,Z,W) = \langle \two' (X,Z),\two' (Y,W)\rangle -\langle \two' (X,W),\two' (Y,Z)\rangle,
\end{equation}
and the Codazzi equations as
\begin{equation}\label{codazzi, original}
\overline{\nabla}_{Y}\two' (X,Z,\nu) - \overline{\nabla}_{X}\two' (Y,Z,\nu) = 0. 
\end{equation}
In the above, $R: \G(T\Sigma)^{\otimes 4} \to \R$ is the Riemann curvature tensor of $(\Sigma,g)$.

Let $\{\partial_1, \partial_2\}$ be a local coordinate frame on $\G(T\Sigma)$ and $\nu=\partial_3$ be the unit normal vector field in $\G(T\Sigma^{\perp})$. Taking $X = \partial_1, Y = \partial_2, Z = \partial_1$, and $W = \partial_2$, the Gauss equation~\eqref{gauss, original} gives us
\begin{equation}\label{gauss'}
    R_{1212} = h_{11}h_{22} - (h_{12})^2,
\end{equation}
where $R_{ijkl}=R(\partial_i, \partial_j, \partial_k, \partial_l)$. For the Codazzi equation~\eqref{codazzi, original}, there are two independent choices $(i, j, k) = (1, 2, 1)$ and $(i, j, k) = (1, 2, 2)$. They lead, respectively, to
\begin{equation}\label{codazzi'}
    \left\{ \begin{array}{ll}
\overline{\nabla}_{\partial_2}\Pi (\partial_1, \partial_1, \partial_3) - \overline{\nabla}_{\partial_1}\Pi (\partial_2, \partial_1, \partial_3) = 0, \\[0.5em]
\overline{\nabla}_{\partial_2}\Pi (\partial_1, \partial_2, \partial_3) - \overline{\nabla}_{\partial_1}\Pi (\partial_2, \partial_2, \partial_3) = 0.
\end{array} \right.
\end{equation}
Define the \emph{second fundamental form}\footnote{A slight abuse of notation here: in do Carmo~\cite{docarmo1992} and many texts, $\two'$ is referred to as the second fundamental form. In this work, we shall work primarily with $\two$, whose components are unknowns of \eqref{eq:Gauss-Codazzi}.} associated to the isometric immersion $\iota$:
\begin{equation}\label{two}
\two:=
\frac{1}{\sqrt{\det g}}\two' = \frac{1}{\sqrt{|g|}}\left[ \begin{array}{ll}
h_{11} & h_{12}\\
h_{21} & h_{22}
\end{array} \right] = \left[ \begin{array}{ll}
L & M\\
M & N
\end{array} \right].    
\end{equation}
In view of the definition of Christoffel symbols $\nabla_{\partial_i}\partial_j= \sum_{\ell=1}^2 \Gamma_{ij}^\ell\partial_\ell$ and Gaussian curvature $$\gauss:= \frac{R_{1212}}{{\det g}},$$ we recast the Gauss--Codazzi equations~\eqref{gauss'}, \eqref{codazzi'} into~\eqref{eq:Gauss-Codazzi}. See \cite[\S 2.1]{li} for detailed derivations.

\subsection{Fundamental theorem of surface theory with weak regularity}\label{sec: fund theorem}

P.~O. Bonnet (1860) proved the following foundational result in~\cite{bonnet} (\emph{cf.} Eisenhart \cite{eisenhart1997} for modern treatments), now known as the ``fundamental theory of surface theory'':
\begin{quotation}
Let $\Omega \subset \R^2$ be a domain with a smooth positive definite $2\times 2$-matrix field $g$. Given any $C^\infty$-solution to the Gauss--Codazzi system~\eqref{eq:Gauss-Codazzi} on the domain $\Omega \subset \R^2$. Then, for any point $p_0\in\Omega$, there exists a neighbourhood $\Omega'\subset\Omega$ containing $p_0$, such that $(\Omega',g)$ admits a $C^\infty$-isometric embedding into $(\R^3,\delta).$  Moreover, this isometric embedding is unique modulo translations and rotations in $\R^3$. 
\end{quotation}

Since Bonnet's seminal work, various improved versions of this result which require weaker regularity assumptions on the metric $g$ and solutions to the Gauss--Codazzi equations had been obtained. See, for example, Malliavin~\cite{malliavin1972}, Choquet-Bruhat--DeWitt-Morette--Dillard-Bleick~\cite{choquet1977},  Blume \cite{blume1989}, P.~G. Ciarlet--Larsonneur~\cite{ciarletlarsonneur2002}, S. Mardare~\cite{mardare2003, mardare2005, mardare2007}, Szopos~\cite{szopos2008}, and P.~G. Ciarlet--C. Mardare~\cite{ciarletmardare2019}, etc.\footnote{The above cited works are mostly motivated by or orientated to problems in applied mathematics, especially in elasticity theory.} Recently, Litzinger \cite{litzinger2021} proved the $W^{2,2}_\loc$-version of the fundamental theorem of surface theory, and Li--Su \cite{lisu} extended to $W^{2,d}_\loc$-isometric immersions, with arbitrary $d=2,3,4,\ldots$, for the Gauss--Codazzi-Ricci system in arbitrary dimensions and codimensions. In fact, one may go slightly beyond the critical regularity $W^{2,d}$ for $d \geq 3$ --- we established in \cite{lisu} isometric immersions in the Morrey space $L_2^{q,d-q}$ for $2 < q \leq d$, provided that the Gauss--Codazzi--Ricci equations have $L^{q,d-q}$-weak solutions.

In this paper, we restrict ourselves to $d=2$ and present the following version of the fundamental theorem of surface theory as in~\cite{litzinger2021}. 
\begin{lemma}\label{lem: surface theory}
Let $\Omega \subset \R^2$ be a domain and $p \in [2,\infty]$. Given a metric tensor $g \in W^{1,p}_\loc(\Omega)$ for $p >2$ and $g \in W^{1,2}_\loc\cap \linf(\Omega)$ for $p=2$. Also, given a weak solution  $(\bar{L},\bar{M},\bar{N}) \in L^p_\loc(\Omega)$ to the Gauss--Codazzi system~\eqref{eq:Gauss-Codazzi}. Then, for any simply-connected subdomain $\Omega' \subset \Omega$, there exists a $W^{2,p}_\loc$-isometric immersion $\iota: (\Omega', g) \to (\R^3,\delta)$ whose second fundamental form is $\overline{\two}=\begin{bmatrix}
    \bar{L} & \bar{M}\\
    \bar{M} & \bar{N}
\end{bmatrix}$. In addition, $\iota$ is unique modulo modification on null sets and Euclidean motions in $\R^3$.

\end{lemma}

%Let us explain why the $W^{2,p}_\loc$-isometric immersions are rigid for $p \geq 2$. Given an isometric immersion~$\iota$ in this regularity class, the Gauss--Codazzi equations are well defined in the sense of distributions. Starting from a Gauss--Codazzi weak solution in $L^p_{\loc}$, we may construct an isometric immersion $\iota'$ by  Lemma~\ref{lem: surface theory}. The uniqueness part of this lemma implies that $\iota = \iota'\circ \zeta$ \emph{a.e.} for some Euclidean rigid motion $\zeta \in \R^3 \rtimes  SO(3)$. This is in stark contrast to the ``wild'' or flexible \(C^{1,\alpha}\)-isometric immersions (with \(0 < \alpha \le \alpha_* < 1\) for some \(\alpha_*\)) constructed by convex integration~\cite{delellis2012}, for which curvature cannot be defined even in the distributional sense.

\section{Invariant regions for Riemann invariants}\label{sec: inv}

\subsection{Parabolic regularisation}

The starting point of the analysis for the Gauss--Codazzi system~\eqref{eq:Gauss-Codazzi} is  the 1D isentropic Euler equations for the Chaplygin gas with source term; \emph{i.e.}, the fluid dynamic formulation~\eqref{rho, m equation}, reproduced below: 
\begin{equation*}
    \begin{cases}
        \p_t \rho + \p_x m = -\tg^{2}_{22}\rho - 2\tg^{2}_{12}m - \tg^2_{11} \frac{m^2-1}{\rho},\\
        \p_t m + \p_x\left(\frac{m^2-1}{\rho}\right) = -\tg^1_{22}\rho -2\tg^1_{12}m - \tg^1_{11}\frac{m^2-1}{\rho}.
    \end{cases}
\end{equation*}
We work with the spacetime domain $$\I_T = [0,T]\times\I \subset \R^2$$ where $T>0$ is an arbitrary finite number and $\I \subset \R$ is an interval. We identify local coordinates $(t,x)$ with $(x,y)$.

To obtain solutions to the above hyperbolic PDE system, we consider its \emph{parabolic regularisation} by adding (artificial) viscous terms:
\begin{equation}\label{parabolic regularisation}
    \begin{cases}
        \p_t \rho^\e + \p_x m^\e = \e\p_{xx}\rho^\e -\tg^{2}_{22}\rho^\e - 2\tg^{2}_{12}m^\e - \tg^2_{11} \frac{(m^\e)^2-1}{\rho^\e},\\
        \p_t m^\e + \p_x\left(\frac{(m^\e)^2-1}{\rho^\e}\right) =  \e\p_{xx}m^\e -\tg^1_{22}\rho^\e -2\tg^1_{12}m^\e - \tg^1_{11}\frac{(m^\e)^2-1}{\rho^\e}.
    \end{cases}
\end{equation}
Given any finite $T>0$, the unique smooth solution $(\rho^\e, m^\e)$ to Eq.~\eqref{parabolic regularisation} exists up to time $T$ by parabolic theory. Our goal is to pass to the limits $(\rho^\e, m^\e) \to (\rho, m)$ by sending $\e \to 0^+$ in some suitable topology to obtain a weak solution $(\rho, m)$ to the above system.

For this purpose, we investigate Eq.~\eqref{parabolic regularisation} subject to the constraints: 
\begin{align}\label{rho-eps, m-esp constraints}
    \rho^\e \geq 0,\qquad -1 \leq m^\e \leq 1 \qquad \text{a.e. on } \I_T.
\end{align}
See Eqs.~\eqref{rho geq 0} and \eqref{lambda pm condition}. As discussed in \S\ref{subsec: Chaplygin gas}, it corresponds to the case that two distinct families of contact discontinuity waves propagate at characteristic speeds, \emph{a.k.a.} Riemann invariants: 
\begin{equation*}
    \begin{cases}
        \lambda^\e_{+} \equiv w^\e = \frac{m^\e + 1}{\rho^\e}, \\\lambda^\e_{-} \equiv z^\e = \frac{m^\e - 1}{\rho^\e}. 
    \end{cases}
\end{equation*}
    
We shall establish $\llinf$-bounds for $(w^\e, z^\e)$ uniformly in $\e>0$. This does not necessarily rule out the degenerate scenario $\rho^\e \to +\infty$; however, it will lead to desirable  bounds for the geometric variables $\left\{\left(L^\e, M^\e, N^\e\right)\right\}$ defined via
\begin{equation}\label{def: Le, Me, Ne}
L^\e:=\gamma \rho^\e, \qquad M^\e = -{\gamma}m^\e,\qquad N^\e=\gamma \left(\frac{(m^\e)^2-1}{\rho^\e}\right),
\end{equation}
which allows us to pass to the weak limits $\left(L^\e, M^\e, N^\e\right)\to \left(\bar{L}, \bar{M}, \bar{N}\right)$ and conclude that $\left(\bar{L}, \bar{M}, \bar{N}\right)$ is a weak solution to the Gauss--Codazzi system~\eqref{eq:Gauss-Codazzi} through the compensated compactness framework (see Lemma~\ref{lem: comp comp framework} below).

The $\llinf$-bounds for $(w^\e, z^\e)$ will be derived using a version of maximum principle for the parabolic PDE system in Riemann invariant coordinates. We recast~\eqref{parabolic regularisation} into the PDE system for $(w^\e,z^\e)$ as follows:
\begin{equation}\label{w,z-regularised eq}
    \begin{cases}
        \p_tw^\e + z^\e\p_x w^\e = \e \frac{\p_x\left((\rho^\e)^2\p_xw^\e\right)}{(\rho^\e)^2} + w^\e \A_2[w^\e;z^\e]-\A_1[w^\e;z^\e],\\
        \p_tz^\e + w^\e\p_x z^\e = \e \frac{\p_x\left((\rho^\e)^2\p_xz^\e\right)}{(\rho^\e)^2} + z^\e\A_2[w^\e;z^\e]-\A_1[w^\e;z^\e],
    \end{cases}
\end{equation}
where the source terms are quadratic in $w^\e$ and $z^\e$:
\begin{equation*}
    \A_k[w^\e;z^\e] := \tg^k_{22} + \tg^k_{12}(w^\e+z^\e) + \tg^k_{11}w^\e z^\e,\qquad k \in \{1,2\}.
\end{equation*} 
The symbols $\tg^i_{jk}$ are given in ~\eqref{modified Gamma}. Note that the source terms $w^\e \A_2[w^\e;z^\e]-\A_1[w^\e;z^\e]$ and $z^\e \A_2[w^\e;z^\e]-\A_1[w^\e;z^\e]$ in \eqref{w,z-regularised eq} are cubic polynomials in $(w^\e, z^\e)$.

\subsection{Invariant regions}
 
One key ingredient of our arguments, as in \cite{chen2010a, li, cao2015}, is the method of invariant regions. It may be regarded as the maximum principle for parabolic PDE systems, and we shall make use of the formulation in Proposition~\ref{propn: inv region} below. A more comprehensive treatment can be found, \emph{e.g.}, in Chueh--Conley--Smoller~\cite{chueh1977}. 

Invariant regions have also played a central role in Morawetz's seminal papers~\cite{morawetz1985, morawetz1995, morawetz2004} on the weak solutions to transonic flow problems. These, in turn, motivated the earlier investigations on hyperbolic Gauss--Codazzi equations by Chen--Slemrod--Wang~\cite{chen2010a}.

\begin{proposition}[Theorem 4.4 in \cite{chueh1977}]\label{propn: inv region}
Consider the PDE system for $U: \R\times\mathbb{R}^m  \to \mathbb{R}^n$:
\begin{equation}\label{general system U}
\p_t U = \epsilon \mathcal{D} \cdot \Delta U + \sum_{j=1}^{m} M^j \p_{x_j}U + \mathcal{S}, 
\end{equation}
where $\mathcal{D}$ is a positive definite $n\times n$ matrix, $\{M^1,\ldots,M^m\}$ are $n\times n$ matrices, and $\mathcal{S}:\R\times\mathbb{R}^m \to \mathbb{R}^n$. The region
\[
\mathscr{R} := \bigcap_{i=1}^{d} \left\{ V \in \mathbb{R}^n : \Phi_i(V) \le 0,\; \Phi_i:\operatorname{Image}(U) \to \mathbb{R} \text{ are smooth functions} \right\}
\]
is an invariant region for \eqref{general system U} for all $\epsilon > 0$, namely that $U_0(x) \equiv U(t_0,x) \in \mathscr{R}$ implies $U(t,x) \in \mathscr{R}$ for all $t \geq t_0$ before the lifespan of the solution, if and only if the following conditions hold:
\begin{enumerate}
\item $\nabla_{U_0} \Phi_i$ is a left eigenvector for $\mathcal{D}$ and each $M^j$, $j=1,2,\ldots,m$;
\item For each vector $\eta \in \mathbb{R}^n$, the Hessian matrix satisfies $D_{U_0}^2 \Phi_i(\eta,\eta) \ge 0$ if $\nabla_{U_0} \Phi_i(\eta)=0$;
\item $\nabla_{U_0} \Phi_i \cdot \mathcal{S}(U_0) \le 0$ for each $i\in \{1,\ldots,d\}$ and any $U_0 \in\partial \mathscr{R} \cap \{\Phi_i=0\}$.
\end{enumerate}
\end{proposition}

Instead of using the invariant regions in the flow speed-phase angle plane investigated in Chen--Slemrod--Wang~\cite{chen2010a} (which is motivated by Morawetz~\cite{morawetz1985, morawetz1995, morawetz2004}) or the fixed square-shaped invariant regions studied in Cao--Huang--Wang \cite{cao2015} and the follow-up work \cite{li}, we consider here a \emph{time-dependent} invariant region $\mathscr{R}(t)$, motivated by the following observations:
\begin{enumerate}
    \item 
As discussed in the Introduction, we expect that the two families of contact discontinuities for the $(\rho, m)$ equation~\eqref{rho, m equation} remain non-interacting with each other. By computing the Riemann invariants --- see \eqref{eigenvalues}, \eqref{lambda pm condition} --- we find that this is equivalent to $w \geq 0$ and $z \leq 0$. Hence, we look for invariant regions in the second quadrant of the $(z,w)$-plane. 

\item 
The source term $\mathcal{S}$ in \eqref{source terms S} is a \emph{cubic} polynomial in $(w,z)$ with coefficients $\tg^i_{jk}$. Previous work \cite{cao2015, li} studied the zero loci of these cubics further transforming to the ``velocity variables'' $\left(u = \frac{w+z}{2}, v=\frac{w-z}{2}\right)$. In contrast, by introducing time-dependent $Q=Q(t)$, we allow the conditions for the invariant regions to be automatically satisfied on the boundaries $\{w=Q(t)\}$ and $\{z=-Q(t)\}$, thus bypassing the difficulties of checking the signs of cubic polynomials.  
\end{enumerate}

\begin{lemma}\label{lem: Riem inv}
Assume that $\widetilde\Gamma^1_{22}=0$ and $\widetilde\Gamma^1_{12}\geq0$ \emph{a.e.} on $\I_T$. Let $Q:[0,T_\star[\to\R$ be the maximal positive solution to the ODE 
\begin{equation}
Q'(t)
=
3K_T
\left(
Q(t)+Q(t)^2+Q(t)^3
\right),
\qquad
Q(0)=Q_0>0,
\label{eq:Q-moving-region}
\end{equation}
where $K_T
:=
\max_{i,j,k\in\{1,2\}}
\left\|
\widetilde\Gamma^i_{jk}
\right\|_{L^\infty(\I_T)}$ and $T_\star$ is the minimum of $T$ and the lifespan of $Q$. Then, for any $t \in [0,T_\star[$, the time-dependent square
\begin{equation*}
\mathcal R(t)
:=
\left\{
(w,z)\in\mathbb R^2:
0\leq w\leq Q(t),\;
-Q(t)\leq z\leq0
\right\}
\end{equation*}
is an invariant region for  the parabolic system~\eqref{w,z-regularised eq}. That is, if $\bigl(w^\e(0,x),z^\e(0,x)\bigr) \in\mathcal R(0)$ for every $x\in\I$, and if the boundary data (when present) take values in
$\mathcal R(t)$, then the solution $\bigl(w^\e(t,x),z^\e(t,x)\bigr)
\in\mathcal R(t)$ for every $t\in [0,T_\star[$. 
\end{lemma}

\begin{proof}[Proof of Lemma~\ref{lem: Riem inv}]

By setting $W^\e(t,x)
:= \frac{w^\e(t,x)}{Q(t)}$ and $Z^\e(t,x):= \frac{z^\e(t,x)}{Q(t)}$, we reduce the problem to showing that the fixed square $[0,1]\times[-1,0]$ in the $(W^\e, Z^\e)$-plane is invariant under the dynamics of \eqref{w,z-regularised eq}. To this end, we check the three conditions of Proposition~\ref{propn: inv region}. As $\Phi_i$ are linear functions, so (2) is automatically satisfied. Also, the outward normals to the four sides of $\mathcal{R}(t)$ are coordinate vectors, and both $\mathcal{D}$ and $M^1$ in \eqref{general system U} are diagonal matrices in this case. Hence (1) is satisfied.

For ease of reference, we denote the source term as follows:
\begin{equation}\label{source terms S}
    \mathcal{S}(w,z) = \begin{bmatrix}
     \mathcal{S}^{(1)}(w,z)\\
     \mathcal{S}^{(2)}(w,z)
    \end{bmatrix} = \begin{bmatrix}
     w\A_2[w;z]-\A_1[w;z]\\
     z\A_2[w;z]-\A_1[w;z]
    \end{bmatrix}.
\end{equation}
Let us \emph{claim} that 
\begin{equation}
\max\left\{\left|\mathcal{S}^{(1)}(w,z)\right|,\,
\left|\mathcal{S}^{(2)}(w,z)\right|\right\}\leq Q'(t)\qquad\text{for } (w,z) \in \mathcal{R}(t).
\label{eq:source-controlled-by-Q}
\end{equation}
To see this, note that $
0\leq w\leq Q(t)$ and $
-Q(t)\leq z\leq 0$ for $(w,z) \in \mathcal{R}(t)$. Then $|w+z|\leq2Q(t)$ and $
|wz|\leq Q(t)^2$, and hence 
\[
|\A_2[w,z]|
\leq
K_T
\left(
1+2Q(t)+Q(t)^2
\right).
\]
Since by assumption $\widetilde\Gamma^1_{22}=0$, we also have
\[
|\A_1[w,z]|
\leq
K_T
\left(
2Q(t)+Q(t)^2
\right).
\] 
The \emph{claim}~\eqref{eq:source-controlled-by-Q} now follows from the definition of $\mathcal{S}^{(i)}$, $i \in \{1,2\}$ in \eqref{source terms S} and the ODE~\eqref{eq:Q-moving-region}.

The sign condition (3) in Proposition~\ref{propn: inv region} can be checked as follows:
\begin{itemize}
    \item
On the boundary $\{w=0\}$, using $\widetilde\Gamma^1_{22}=0$, we have $\A_1[0,z] = \widetilde\Gamma^1_{12}z,
$ and hence
$\mathcal{S}^{(1)}(0,z)
=-\A_1[0,z]=
-\widetilde\Gamma^1_{12}z
\geq0$ whenever $z\leq0$.

    \item 
Similarly, on the boundary $\{z=0\}$, we have $\A_1[w,0] =
\widetilde\Gamma^1_{12}w$,
so that $\mathcal{S}^{(2)}(w,0) = -\A_1[w,0] =
-\widetilde\Gamma^1_{12}w
\leq0 $ whenever $w\geq0$.

\item 
On the moving boundary $\{w=Q(t)\}$, denote the corresponding $\Phi_i$ function as 
\[
\Phi_+(t,w,z)
:=
w-Q(t).
\]
We deduce from~\eqref{eq:source-controlled-by-Q} that
\[
\partial_t\Phi_+
+
\nabla_{w,z}\Phi_+\cdot \mathcal{S}
=
-Q'(t)+\mathcal{S}^{(1)}(Q(t),z)
\leq0.
\]

\item 
Finally, on the moving boundary $\{z=-Q(t)\}$, set 
\[
\Phi_-(t,w,z)
:=
-z-Q(t).
\]
Then, again by \eqref{eq:source-controlled-by-Q}, one has
\[
\partial_t\Phi_-
+
\nabla_{w,z}\Phi_-\cdot \mathcal{S}
=
-Q'(t)-\mathcal{S}^{(2)}(w,-Q(t))
\leq0.
\]
\end{itemize}
This verifies (3) and hence proves Proposition~\ref{propn: inv region}.   \end{proof}

When specialising to diagonal metrics $g=E(x,y)dx^2+G(x,y)dy^2$, we have $\tg^1_{22}=-\frac{\p_xG}{2E}$, which equals zero if and only if $G=G(y)$. On the other hand, $\tg^1_{12} = \frac{1}{2}\p_y\left[\log (E\gamma)\right]$.  

%The lifespan $T_\star$ can be taken arbitrarily large if $Q(0)$ is arbitrarily small. Indeed, the ODE~\eqref{eq:Q-moving-region} becomes $Q_+' \leq CQ_{+}^3$ for $Q_+(t):=1+Q(t)$ and $K_T'$ is a universal constant times $K_T$. Then $Q_+(t) \leq \frac{Q_+(0)}{\sqrt{1-2K_T' Q_+^2(0) t}}$.

\begin{remark}\label{remark: initial data}
The conditions on invariant regions in Lemma~\ref{lem: Riem inv} entail the following for the initial data $\left(L_0, M_0, N_0\right):=\left(L,M,N\right)\big|_{\{y=0\}}$:
\begin{align*}
L_0 > 0,\quad \left|M_0\right| \leq \gamma,\quad 1 \pm \frac{M_0}{\gamma} \leq Q\frac{L_0}{\gamma}.
\end{align*}

Meanwhile, note from the ODE~\eqref{eq:Q-moving-region} that the lifespan
\begin{align*}
    T_\star \gtrsim \log\left(\frac{1}{Q_0}\right)
\end{align*}
modulo a constant depending only on $K_T$, provided that $T$ (in $\I_T=\I\times[0,T]$) is  larger than the right-hand side. In particular, if the initial datum $Q_0$ can be chosen arbitrarily small, then the lifespan is arbitrarily large.
\end{remark}

\section{Uniform estimates via entropy analysis}\label{sec: entropy}

In this section, we shall establish the uniform estimate for $\{L^\e\}$ in the space $L^\infty_t L^p_x$ for $p \in [1,\infty[$ on arbitrarily large rectangular domains, while retaining the possibility for $L^\e$ to become unbounded. For $p \geq 2$, we also establish the uniform derivative estimate $\left\{\sqrt{\e}(\p_x L^\e), \sqrt{\e}(\p_x M^\e)\right\}$ in $(L^2_tL^2_x)_\loc$. Here, $L^\e$ and $M^\e$ are approximate solutions to the Gauss--Codazzi equations, defined through the fluid variables $(\rho^\e, m^\e)$. These bounds will play a crucial role in the construction of weak solutions to the Gauss--Codazzi system~\eqref{eq:Gauss-Codazzi} via the vanishing viscosity method in the subsequent section.

To this end, we exploit the fluid dynamical formulation of the Gauss--Codazzi equations and derive estimates for the fluid variables $(\rho^\e, m^\e)$. This is done in two steps: %we first bound for $\rho^\e$ by entropy considerations via Riemann invariants, and then estimate the first derivatives of $\rho^\e$ and $m^\e$ by energy estimates. 

\begin{itemize}
    \item   
The $L^\infty_t L^p_x$-bounds for $\{\rho^\e\}$ are established via an entropy analysis applied to~\eqref{w,z-regularised eq} for the regularised Riemann invariant coordinates $(w^\e, z^\e)$, reproduced below:
\begin{equation}\label{w,z equation; entropy}
    \begin{cases}
        \p_tw^\e + z^\e\p_x w^\e = \e \frac{\p_x\left((\rho^\e)^2\p_xw^\e\right)}{(\rho^\e)^2} + w^\e \A_2[w^\e;z^\e]-\A_1[w^\e;z^\e],\\
        \p_tz^\e + w^\e\p_x z^\e = \e \frac{\p_x\left((\rho^\e)^2\p_xz^\e\right)}{(\rho^\e)^2} + z^\e\A_2[w^\e;z^\e]-\A_1[w^\e;z^\e].
    \end{cases}
\end{equation}
Also recall the source terms:
\begin{equation*}
    \A_k[w^\e;z^\e] := \tg^k_{22} + \tg^k_{12}(w^\e+z^\e) + \tg^k_{11}w^\e z^\e,\qquad k \in \{1,2\}.
\end{equation*} 

\item 
The uniform $L^2_tL^2_x$-bounds for $\left\{\sqrt{\e}(\p_x \rho^\e), \sqrt{\e}(\p_x m^\e)\right\}$ will be further  deduced from energy estimates for the $\left(\rho^\e, m^\e\right)$ equation~\eqref{parabolic regularisation}, reproduced below:
\begin{equation}\label{rho, m eq, new}
    \begin{cases}
        \p_t \rho^\e + \p_x m^\e = \e\p_{xx}\rho^\e -\tg^{2}_{22}\rho^\e - 2\tg^{2}_{12}m^\e - \tg^2_{11} \frac{(m^\e)^2-1}{\rho^\e},\\
        \p_t m^\e + \p_x\left(\frac{(m^\e)^2-1}{\rho^\e}\right) =  \e\p_{xx}m^\e -\tg^1_{22}\rho^\e -2\tg^1_{12}m^\e - \tg^1_{11}\frac{(m^\e)^2-1}{\rho^\e}.
    \end{cases}
\end{equation}
\end{itemize}

Before proceeding to the proof,  recall that the Riemann invariants $(w^\e,z^\e)$ are related to the fluid variables $(\rho^\e, m^\e)$ by
\begin{align*}
    w^\e = \frac{m^\e+1}{\rho^\e},\qquad z^\e =\frac{m^\e-1}{\rho^\e}.
\end{align*}
In turn, the fluid variables $(\rho^\e, m^\e)$ and the geometric variables $(L^\e,M^\e,N^\e)$ are related by
\begin{align*}
    \rho^\e = \frac{L^\e}{\gamma}, \qquad m^\e = -\frac{M^\e}{\gamma},\qquad n^\e:=\frac{\left(m^\e\right)^2-1}{\rho} = \frac{N^\e}{\gamma},
\end{align*}
where $\gamma = \sqrt{-\gauss}$. In the next section, we shall show that the weak limits as $\e \to 0$ of $(L^\e,M^\e,N^\e)$ are the weak solutions to the Gauss--Codazzi system~\eqref{eq:Gauss-Codazzi}.

\begin{proposition}\label{propn: entropy analysis}
    Consider the PDE~\eqref{w,z equation; entropy} for the Riemann invariants $(w^\e,z^\e)$ and the PDE~\eqref{rho, m eq, new} for the fluid variables $(\rho^\e, m^\e)$ over the spacetime domain $\I_T := [0,T] \times \I$, where $\I\subset \R$ is an interval, subject to the boundary condition (or the far-field condition if $\I$ is unbounded)
\begin{equation}\label{bc for rho, m}
   \left[\p_x \rho^\e(t,\bullet) ,  m^\e(t,\bullet) \right]^\top\Big|_{\p\I}=[0,0]^\top\qquad \text{for a.e. } t \in [0,T].
\end{equation}

\begin{enumerate}
    \item 
Fix $p \in [1,\infty[$. Assume that the initial entropy  is integrable on $\I$:
$$\rho^\e\big((w^\e)^{1-p} + (-z^\e)^{1-p}\big)\Big|_{t=0} \in L^1(\I),$$ and $0 \leq w^\e \leq Q$, $-Q \leq z^\e \leq 0$ for some $Q$ independent of $\e$. Then $\left\|\rho^\e\right\|_{L^\infty_t L^p_x(\I_T)}$ is bounded by a constant depending only on $p$, $T$, $Q$, and the $W^{3,\infty}$-norm of the metric $g$.
\item 
Fix $p=2$ in (1) above. Suppose in addition that  $\tg^{1}_{22}=0$,  $\tg^1_{12}\geq 0$, as well as $0\leq w^\epsilon\leq Q$ and $-Q\leq z^\epsilon\leq 0$ \emph{a.e.} on $\I_T$ for some constant $Q>0$ independent of $\e$. Then $\left\{\e \left(\p_x\rho^\e\right)^2 \right\}$ and $\left\{\e \left(\p_x m^\e\right)^2 \right\}$ are bounded in $L^1(\I_T)$ by a constant depending only on $p$, $T$, $Q$, and the $W^{3,\infty}$-norm of the metric $g$.  
\end{enumerate}
\end{proposition}

    The boundary conditions~\eqref{bc for rho, m} are equivalent to 
    \begin{equation*}
           \left[\p_x\left(\frac{L^\e}{\gamma}\right)(t,\bullet) , \,M^\e(t,\bullet) \right]^\top\Big|_{\p\I}=[0,0]^\top\qquad \text{for a.e. } t \in [0,T].
    \end{equation*}
    The assumption $\rho^\e\big((w^\e)^{1-p} + (-z^\e)^{1-p}\big)\big|_{\{t=0\} }\in L^1(\I)$ allows concentration for the initial density, that is,  $\rho^\e(0,x)=\infty$ for some $x \in \I$, and/or the degeneracy of the Riemann invariants, \emph{i.e.}, $w^\e(0,x)=0$ or $z^\e(0,x)=0$ for some $x \in \I$. The pointwise boundedness assumptions  $0 \leq w^\e \leq Q$, $-Q \leq z^\e \leq 0$  are valid in the invariant regions constructed in the previous sections.

    In (2) above, we specialise to $p=2$ and work in the setting of Lemma~\ref{lem: Riem inv}, noting that $\tg^1_{22}=0$ is ensured by considering $G=G(y)$ for the metric component. Here the integrability condition for the initial entropy becomes \begin{equation}\label{new condition, initial data}
    \sup_{\epsilon>0}
\int_\I
\frac{(\rho^\epsilon_0)^2}
     {1-(m^\epsilon_0)^2}\,\dd x
<\infty.
\end{equation}

A key remark is in order:
\begin{remark}\label{remark: finite domain}
In view of the definition of $(w^\e, z^\e)$ and the assumptions $0 \leq w^\e\leq Q$, $-Q \leq z^\e \leq 0$, we have the lower bound $\rho \geq Q^{-1}>0$. Hence, the integrability condition for the initial entropy implies that $\I$ must be bounded here. Later, in \S\ref{sec: new} we shall tackle the case $\I=\R$ by exploiting the \emph{relative entropy} \emph{\`{a} la} Dafermos~\cite{dafermos2016} instead of the entropy $\eta$ in the proof below.

%In \S\ref{sec: conclusion} we extend this argument to $\I=\R$ by using the \emph{relative entropy} \emph{\`{a} la} Dafermos~\cite{dafermos2016} instead of the entropy $\eta$ in the proof below. Nevertheless, in general we do not have a constant steady-state solution to the Gauss--Codazzi equations lying in the invariant region, so the relative entropy argument requires additional restrictions on the signs of certain $\tg^i_{jk}$. %This is why only six out of the sixteeen families of metrics in Table~\ref{table: 2} are shown to admit global, rather than semiglobal, isometric immersions.
\end{remark}

\begin{proof}[Proof of Proposition~\ref{propn: entropy analysis}]  We prove the two statements by entropy and energy analysis, respectively. For notational convenience, we drop the superscript ${}^\e$ in $\rho^\e$, $m^\e$, $w^\e$, $z^\e$..., and we also abbreviate $\A_k \equiv \A_k[w^\e;z^\e]$ for $k \in \{1,2\}$ throughout this proof.

\smallskip
\noindent
\underline{Proof of (1).} 
Let $f$, $g:\R\to\R$ be convex functions to be specified. 
Define the entropy $\eta$ and the entropy flux $q$ as follows: 
\begin{equation}\label{entropy-flux pair}
    \begin{cases}
    \eta := \rho\big(f(w)+g(z)\big),\\
    q := \rho\big(zf(w)+wg(z)\big).
    \end{cases}
\end{equation}
From \eqref{w,z equation; entropy}, we deduce the balance law:
\begin{equation}\label{eta, q eq}
    \p_t\eta + \p_xq = \e\p_{xx}\eta - \e\rho\big[f''(w)(\p_xw)^2 + g''(z) (\p_xz)^2 \big] + \Sigma,
\end{equation}
where the flux associated with the source term is
\begin{equation*}
   \Sigma:= \rho\big[wf'(w)+zg'(z)-f(w)-g(z)\big]\A_2  - \rho\big[f'(w)+g'(z)\big]\A_1. 
\end{equation*} 

Next, for fixed $p \in [1,\infty[$, we specialise to
\begin{equation*}
\begin{cases}
    f(w):=w^{1-p}, \\
    g(z):=(-z)^{1-p}.
\end{cases}
\end{equation*}
As $-1 \leq m \leq 1$, $w\geq 0$, and $z \leq 0$, it holds at every point in $\I_T$ that
\begin{align*}
    \eta &= \rho\big(w^{1-p} + (-z)^{1-p}\big)\\
    &= \rho^p\Big((m+1)^{1-p} + (1-m)^{1-p}\Big) \geq \rho^p.
\end{align*}

In view of the boundary condition~\eqref{bc for rho, m}, we have that $q\big|_{\p\I}\equiv 0$ and  $(\p_x\eta)\big|_{\p\I}\equiv 0$. We then deduce by integrating Eq.~\eqref{eta, q eq} over $\I_t$ that
\begin{align}\label{L^p estimate for rho}
  &\int_\I \eta(t,x)\,\dd x  + \e p(p-1)\iint_{\I_t}\rho\left\{w^{-1-p}(\p_xw)^2+(-z)^{-1-p}(\p_xz)^2\right\}\,\dd x \,\dd s \nonumber\\
  &\quad\leq \int_\I \eta(0,x)\,\dd x  +  \iint_{\I_t} \Sigma (s,x)\,\dd x\,\dd s \qquad\text{for any } t \in [0,T].
\end{align}
Thanks to the definition of $\Sigma$, $\eta$, $\A_1$, and $\A_2$, we have the pointwise bound:
\begin{align*}
    \Sigma &= -p\rho \Big( w^{1-p} + (-z)^{1-p} \Big)\A_2 + (p-1)\rho \Big(w^{-p}-(-z)^{-p}\Big)\A_1\\
    &\leq C \rho\big(w^{1-p} + (-z)^{1-p}\big) \equiv C\eta
\end{align*}
for a constant $C$ depending only on $p$, $\|w\|_{\linf(\I_T)}$, $\|z\|_{\linf(\I_T)}$, $\|g\|_{W^{1,\infty}(\I_T)}$,  and $\|\log\gamma\|_{W^{1,\infty}(\I_T)}$. Here we have used $\tg^1_{12}\geq0$ and $\big(w^{-p} - (-z)^{-p}\big)(w+z)\leq 0$.

Therefore, by the pointwise bounds $\rho^p \leq \eta$, $\Sigma \leq C\eta$ and Gr\"{o}nwall's inequality, we have 
\begin{align*}
    \|\rho\|_{L^\infty_t L^p_x(\I_T)} &\leq C\left(p,T,\|\eta_0\|_{L^1(\I)}, \|w\|_{\linf(\I_T)}, \|z\|_{\linf(\I_T)} \right).
\end{align*}
In particular, the constant is uniform in $\e$. This proves (1).

\smallskip
\noindent
\underline{Proof of (2).} 
We again suppress the superscript ${}^\epsilon$. Recall that $w=\frac{1+m}{\rho}$, $z=\frac{m-1}{\rho}$. We shall make use of the entropy pair corresponding to $f(w)=w^{-1}$ and $g(z)=(-z)^{-1}$; \emph{i.e.}, as in (1) above with $p=2$. Here, the entropy-entropy flux pair is given by
\begin{align}\label{eq:p2-entropy} 
&\eta_2(\rho,m)=\rho\bigl(f(w)+g(z)\bigr) =\rho\left(\frac1w+\frac1{-z}\right)=\frac{2\rho^2}{1-m^2},\nonumber\\
&
q_2(\rho,m) :=\rho\bigl(zf(w)+wg(z)\bigr) =\rho\left(\frac zw+\frac{w}{-z}\right)
 =\frac{4\rho m}{1-m^2}.
\end{align}

Denote $U=[\rho,m]^\top$. The viscous entropy identity takes the
form
\begin{equation}\label{eq:p2-entropy-identity}
\partial_t\eta_2+\partial_xq_2
=
\epsilon\partial_{xx}\eta_2
-
\epsilon
\left\langle
D^2\eta_2(U)\partial_xU,\partial_xU
\right\rangle
+\Sigma_2, 
\end{equation}
where
\begin{equation*}
\Sigma_2
=
-2\rho
\left(
 w^{-1}+(-z)^{-1}
\right)\A_2
+
\rho
\left(
 w^{-2}-(-z)^{-2}
\right)\A_1.
\label{eq:p2-source}
\end{equation*}
In view of \eqref{eq:p2-entropy}, we have
\begin{equation}
\Sigma_2
=
-2\eta_2 \A_2
+
\rho
\left(
 w^{-2}-(-z)^{-2}
\right)\A_1.
\label{eq:p2-source-2}
\end{equation}

We first establish the uniform convexity of $\eta_2$ on the
invariant region. Indeed,
\begin{equation*}
D^2\eta_2(\rho,m)
=
\begin{bmatrix}
\dfrac4\delta
&
\dfrac{8\rho m}{\delta^2}
\\
\dfrac{8\rho m}{\delta^2}
&
\dfrac{4\rho^2(1+3m^2)}{\delta^3}
\end{bmatrix}\qquad\text{where } \delta:=1-m^2.
\label{eq:p2-hessian}
\end{equation*}
We thus have
\begin{equation}
\det D^2\eta_2
=
\frac{16\rho^2}{\delta^3}>0\quad\text{and}\quad \operatorname{tr}D^2\eta_2
=
\frac4\delta
+
\frac{4\rho^2(1+3m^2)}{\delta^3}.
\label{eq:p2-hessian-det-trace}
\end{equation}

By assumption of the proposition, the Riemann invariants $w\equiv w^\e$ and $z\equiv z^\e$ are uniformly bounded in the invariant regions. Note that $w-z=\frac2\rho
\leq 2Q$, from which it follows that
\begin{equation}
\rho\geq
\rho_*:=\frac{1}{Q}>0.
\label{eq:rho-lower-bound}
\end{equation}
The smaller eigenvalue of a positive definite $2\times2$ matrix $H$ satisfies $$\lambda_{\min}(H)
\geq
\frac{\det H}{\operatorname{tr}H}.$$ Thus, by the bounds in~\eqref{eq:p2-hessian-det-trace} and \eqref{eq:rho-lower-bound}, we have that
\begin{align*}
\lambda_{\min}\bigl(D^2\eta_2(\rho,m)\bigr)
&\geq
\frac{4\rho^2}
     {\delta^2+\rho^2(1+3m^2)}
\notag\\
&\geq
\frac{4\rho^2}{1+4\rho^2}
\geq
\frac{4\rho_*^2}{1+4\rho_*^2}
=:c_*>0,
\label{eq:p2-uniform-convexity}
\end{align*}
and hence
\begin{equation}
\left\langle
D^2\eta_2(U)\partial_xU,\partial_xU
\right\rangle
\geq
c_*
\left(
 |\partial_x\rho|^2+|\partial_xm|^2
\right).
\label{eq:p2-dissipation}
\end{equation}

For the source term, let us \emph{claim} that \begin{equation}
\Sigma_2\leq C\eta_2.
\label{eq:p2-source-final}
\end{equation}
Indeed, thanks to the invariant region established in Lemma~\ref{lem: Riem inv}, $0<w \leq Q$ and $0<-z \leq  Q$ \emph{a.e.} on $\I_T$. Also by assumption we have $\widetilde{\Gamma}^{1}_{22}=0$. Note that
\begin{align*}
\A_1 =
\widetilde{\Gamma}^{1}_{12}(w+z)
+
\widetilde{\Gamma}^{1}_{11}wz,   
\end{align*}
and hence the part containing
$\widetilde{\Gamma}^{1}_{12}$ has a favourable sign:
\begin{align*}
\left(w^{-2}-(-z)^{-2}\right)
\widetilde{\Gamma}^{1}_{12}(w+z) =
-\widetilde{\Gamma}^{1}_{12}
\frac{(w+z)^2(w-z)}{w^2z^2}
\leq0,
\label{eq:A1-good-sign}
\end{align*}
where we have used
$\widetilde{\Gamma}^{1}_{12}\geq0$. On the other hand,
\begin{align*}
\left|
\left(w^{-2}-(-z)^{-2}\right)
\widetilde{\Gamma}^{1}_{11}wz\right| =
\left|
\widetilde{\Gamma}^{1}_{11}
\right|
\left|
\frac wz-\frac zw
\right| \leq CQ
\left(
 \frac1w+\frac{1}{-z}
\right).
\end{align*}
We may thus bound the entropy $\eta_2
=
\rho
\left(
 \frac1w+\frac1{-z}
\right)$ by
\begin{equation*}
\rho
\left(
w^{-2}-(-z)^{-2}
\right)\A_1
\leq
C\eta_2.
\label{eq:A1-source-bound}
\end{equation*}
In view of the expression~\eqref{eq:p2-source-2} for $\Sigma_2$ and the bound~$|\A_2|\leq C$ \emph{a.e.} on $\I_T$, we may now conclude the \emph{claim}~\eqref{eq:p2-source-final}.

We now integrate \eqref{eq:p2-entropy-identity} over $\I$. The
boundary conditions on $\rho$ and $m$ imply that 
$q_2\big|_{\partial \I}=0.$ Moreover, we have $\partial_x\eta_2\big|_{\partial \I}=0$ by the identity
\[
\partial_x\eta_2
=
\frac{4\rho}{1-m^2}\partial_x\rho
+
\frac{4\rho^2m}{(1-m^2)^2}\partial_xm.
\]
Therefore, from the bounds~\eqref{eq:p2-dissipation} and 
\eqref{eq:p2-source-final}, we infer that
\begin{equation*}
\frac{d}{dt}
\int_\I\eta_2(t,x)\,\dd x
+
c_*\epsilon
\int_\I
\left(
 |\partial_x\rho|^2
 +
 |\partial_xm|^2
\right)\,\dd x 
\leq
C\int_\I\eta_2(t,x)\,\dd x.
\label{eq:p2-energy-ineq}
\end{equation*}
By Gr\"onwall's inequality,
\begin{equation*}
\sup_{0\leq t\leq T}
\int_\I\eta_2(t,x)\,\dd x
\leq
e^{CT}
\int_\I\eta_2(0,x)\,\dd x.
\label{eq:p2-gronwall}
\end{equation*}
From the previous two estimates we infer that
\begin{equation}
\epsilon
\int_0^T\int_\I
\left(
 |\partial_x\rho|^2
 +
 |\partial_xm|^2
\right)\,\dd x\,\dd t
\leq
C_T
\int_\I\eta_2(0,x)\,\dd x.
\label{eq:p2-gradient-final}
\end{equation}
The right-hand side is bounded uniformly in $\epsilon$ by
assumption. Restoring the superscript ${}^\epsilon$, we have established that
$\left\{
\sqrt{\epsilon}\,\partial_x\rho^\epsilon
\right\}_{\epsilon>0}$ and $\left\{
\sqrt{\epsilon}\,\partial_xm^\epsilon
\right\}_{\epsilon>0}$ are bounded in $L^2(\I_T)$ uniformly in $\e$. 

The calculation above may first be carried out for smooth solutions satisfying $w^\epsilon>0$ and $z^\epsilon<0$. General initial data allowing $w^\epsilon=0$ or
$z^\epsilon=0$ are handled by approximation with smooth initial
data lying strictly inside the square-shaped invariant region $\mathscr{R}$, whose existence under the assumptions of this proposition is ensured by Lemma~\ref{lem: Riem inv}. All the
estimates above are independent of the distance to the boundary
of the invariant region. Hence, the estimate pertains in the limiting
process.  This proves (2).   \end{proof}

\begin{comment}
An intriguing characterisation for the degenerate scenario $\rho = + \infty$ is available: 
\begin{lemma}\label{lem: concentration}
Suppose that $\rho \nearrow \infty$ as one approaches an interior point $P_0$ in the domain; or, equivalently, $L=\gamma \rho \nearrow \infty$. Then $N$ decays to zero and $M$ approaches $\pm 1$ at some rate faster than any polynomial of $\rho^{-1}$. In fact,
\begin{equation*}
\text{$-N \lesssim \frac{1}{\rho^\ell}$ and one of $1 \pm \frac{M}{\gamma} \lesssim \frac{1}{\rho^\ell}$ for any $\ell \in \mathbb{N}$ near $P_0$.}
\end{equation*} 
\end{lemma}

\begin{proof}[Proof of Lemma~\ref{lem: concentration}]
We compute $N$ in two different ways:
\begin{equation*}
    -N \stackrel{\text{(a)}}{=} \gamma \cdot \frac{1-m^2}{\rho}  \stackrel{\text{(b)}}{=}  \gamma \lambda_+ \lambda_-\rho.
\end{equation*}
Since $-1 \leq m \leq 1$ and $\lambda_\pm$ are bounded, by (a) we have $-N \lesssim \mathcal{O}(\rho^{-1})$ near $P_0$. Then, in view of (b), one of $\lambda_\pm$ degenerates at order $\mathcal{O}\left({\rho^{-2}}\right)$, so $1 \pm m \lesssim \mathcal{O}(\rho^{-1})$  near $P_0$. Then $-N\lesssim \mathcal{O}(\rho^{-2})$ by (a), and hence $1 \pm m \lesssim \mathcal{O}(\rho^{-2})$ by (b)... Keep iterating this argument to conclude.   \end{proof}

\end{comment}

\section{Relative entropy analysis on infinite strip}
\label{sec: new}

In this section, we shall first prove Theorem~\ref{thm:global-relative-entropy},  reproduced below:

\begin{theorem*}
Let
$U^\e = [\rho^\e, m^\e]^\top$ be a smooth solution in $[0,T]\times\mathbb{R}$ to the balance law
\begin{equation}
\partial_tU^\varepsilon
+
\partial_xF(U^\varepsilon)
=
\varepsilon\partial_{xx}U^\varepsilon
+
S(U^\varepsilon),
\label{eq:viscous-balance-law-global}
\end{equation}
where $F(\rho,m)
=\begin{bmatrix}
    m\\
    \dfrac{m^2-1}{\rho}
\end{bmatrix}$ and $S(\rho,m)
=
\begin{bmatrix}
-a(t)\rho-2b(t)m-c(t)\left(\dfrac{m^2-1}{\rho}\right)\\
-2k(t)m
\end{bmatrix}$. Assume that $a,b,c,k\in L^\infty(]0,T[)$, $k(t)\geq 0$ for \emph{a.e.} $t\in[0,T]$, and $\rho^\e\geq\rho_*>0$, $-1<m^\e<1$ on $[0,T]\times\mathbb{R}$ for $\rho_*$ independent of $\e$. Set $\overline U(t):=[R(t),0]^\top$, where $R$ is the solution to the ODE: 
\begin{equation}
R'(t)
=
-a(t)R(t)+\frac{c(t)}{R(t)}
\label{eq:reference-density-ode}
\end{equation}
that satisfies $0<R_*\leq R(t)\leq R^*<\infty$ for all $t \in [0,T]$. Then for given $p\in[2,\infty[$ we have  
\begin{equation}
\sup_{\e>0}
\sup_{0\leq t\leq T}
\int_{\mathbb{R}}
H_p\bigl(
U^\e(t,x)\mid\overline U(t)
\bigr)\,\dd x
<\infty,
\label{eq:global-relative-entropy-bound}
\end{equation}
provided that $\sup_{\e>0}
\int_{\mathbb{R}}
H_p\bigl(
U^\e_0(x)\mid\overline U(0)
\bigr)\,\dd x
<\infty$. In particular, 
\begin{equation}
\sup_{\e>0}
\left\|
\rho^\e-R
\right\|_{L^\infty(0,T;L^p(\mathbb{R}))}
<\infty.
\label{eq:global-density-Lp-bound}
\end{equation}  
Moreover, we have the gradient estimate:
\begin{align}\label{rel entropy, gradient estimate}
    \sup_{\e>0}\left\{\e\iint_{[0,T]\times\R} |\p_xU^\e|^2 \,\dd x\,\dd t\right\} < \infty.
\end{align}
\end{theorem*}

 The relative entropy in consideration is defined as follows, for reasons that will become transparent along the proof:
\begin{align*}
&H_p\left(U=[\rho,m]^\top\mid\overline U=[R,0]^\top\right)
\\
&\quad:= 
\rho^p \left[(1+m)^{1-p} + (1-m)^{1-p}\right]
-
2R^p
-
2pR^{p-1}(\rho-R).
\end{align*}
This theorem establishes the uniform $L^\infty_tL^p_x$-bound for $\{(\rho^\e, m^\e)\}$ and uniform $L^2_tL^2_x$-bound for $\{(\sqrt{\e}\p_x\rho^\e, \sqrt{\e}\p_xm^\e)\}$ over the infinite strip $\R\times[0,T]$, which extends the analysis in \S\ref{sec: entropy} that applies only to bounded domains (Remark~\ref{remark: finite domain}). The initial data satisfying the required hypothesis
$\sup_{\e>0}
\int_{\mathbb{R}}
H_p\bigl(
U^\e_0(x)\mid\overline U(0)
\bigr)\,\dd x
<\infty$ are not difficult to construct: for example, one may take \[
m^\varepsilon_0\equiv0,
\qquad
\rho^\varepsilon_0-R(0)\in C_c^\infty(\mathbb{R}),
\]
with the data lying strictly in the invariant region.  As a caveat, note that the assumptions $\rho^\e_0-R(0)\in L^1(\mathbb{R})\cap L^p(\mathbb{R})$ and $m^\varepsilon_0\in L^1(\mathbb{R})$ alone cannot ensure the finiteness of the initial relative entropy, for $h_p(m)$ in the proof below blows up as $m\to\pm 1$.

\begin{proof}[Proof of Theorem~\ref{thm:global-relative-entropy}]
Fix any $p \in [2,\infty[$. For notational convenience, we suppress the superscript
$\varepsilon$ throughout the proof. We divide our arguments in eight steps below.

\smallskip
\noindent
{\bf Step~1.} Let us denote
\begin{equation*}
h_p(m)
:=
(1+m)^{1-p}+(1-m)^{1-p}
\label{eq:hp-definition}
\end{equation*}
(noting that $h_p(0)=2$,  $h_p'(0)=0$) and  
\begin{equation*}
\eta_p(\rho,m)
:=
\rho^p h_p(m).
\label{eq:global-entropy-definition}
\end{equation*}
This is the entropy corresponding to
$f(w)=w^{1-p}$ and $g(z)=(-z)^{1-p}$, as in the arguments for Proposition~\ref{propn: entropy analysis}. Also, observe that $\overline U(t)=[R(t),0]^\top$ is a spatially homogeneous solution of the balance law~\eqref{eq:viscous-balance-law-global}. Indeed, in this case we have $$S(\overline U) =
\left[-a(t)R(t)+\frac{c(t)}{R(t)},\,
0\right]^\top,$$ which equals $
\partial_t\overline U$
in view of the ODE~\eqref{eq:reference-density-ode} for $R(t)$.

\smallskip
\noindent
{\bf Step~2.} Let us introduce the relative entropy:
\begin{equation}
H_p(U\mid\overline U)
:=
\eta_p(U)-\eta_p(\overline U)
-\nabla\eta_p(\overline U)
 \cdot(U-\overline U).
\label{eq:relative-entropy-definition}
\end{equation}
Our choice of the entropy $\eta_p$ in this proof yields that 
\begin{align*}
H_p(U\mid\overline U)
&=
\rho^p h_p(m)
-
2R^p
-
2pR^{p-1}(\rho-R)
\notag\\
&=
\rho^p\bigl(h_p(m)-2\bigr)
+
2\Phi_p(\rho\mid R),
\label{eq:relative-entropy-decomposition}
\end{align*}
where
\begin{equation*}
\Phi_p(\rho\mid R)
:=
\rho^p-R^p-pR^{p-1}(\rho-R).
\label{eq:scalar-relative-entropy}
\end{equation*}
By convexity of $s\mapsto s^p$, one has $h_p(m) \geq 2$ and
\begin{equation}
c|\rho-R|^2(\rho+R)^{p-2}
\leq
\Phi_p(\rho\mid R)
\leq
C|\rho-R|^2(\rho+R)^{p-2},
\label{eq:scalar-relative-entropy-coercivity}
\end{equation}
where the constants $c,C$ depend only on
$p,\rho_*,R_*$, and $R^*$.

We also have the bound: 
\begin{equation}
H_p(U\mid\overline U)
\geq
c|\rho-R|^p.
\label{eq:relative-entropy-controls-Lp}
\end{equation}
To see this, as $h_p$ is an even and strictly convex function satisfying $h_p'(0)=0$, there exists $c_p>0$ such that $h_p(m)-2\geq c_p m^2$  for  $-1<m<1$. This together with~\eqref{eq:scalar-relative-entropy-coercivity} leads to
\begin{equation*}
H_p(U\mid\overline U)
\geq
c\left[
|\rho-R|^2(\rho+R)^{p-2}
+
\rho^p m^2
\right],
\label{eq:relative-entropy-coercivity}
\end{equation*}
which clearly implies \eqref{eq:relative-entropy-controls-Lp}.

\smallskip
\noindent
{\bf Step~3.}  Let $q_p$ denote the entropy flux corresponding to $\eta_p$ as in the proof of Proposition~\ref{propn: entropy analysis}. We define the relative entropy flux by
\begin{equation}
Q_p(U\mid\overline U)
:=
q_p(U)-q_p(\overline U)
-
\nabla\eta_p(\overline U)
\cdot
\bigl(
F(U)-F(\overline U)
\bigr).
\label{eq:relative-entropy-flux}
\end{equation}
Since $\overline U$ is independent of $x$ and satisfies the ODE $\partial_t\overline U=S(\overline U)$, we have the following associated balance law for the relative entropy-entropy flux pair:
\begin{align}
\partial_tH_p(U\mid\overline U)
+
\partial_xQ_p(U\mid\overline U)
={}&
\e
\partial_{xx}H_p(U\mid\overline U) -
\e
D^2\eta_p(U)
[
\partial_xU,\partial_xU
]
+
\mathcal R_p,
\label{eq:relative-entropy-identity}
\end{align}
with the source term
\begin{align}
\mathcal R_p
={}&
\bigl(
\nabla\eta_p(U)-\nabla\eta_p(\overline U)
\bigr)
\cdot
\bigl(
S(U)-S(\overline U)
\bigr)
\notag\\
&+
\Big[
\nabla\eta_p(U)
-
\nabla\eta_p(\overline U)
-
D^2\eta_p(\overline U)(U-\overline U)
\Big]
\cdot S(\overline U).
\label{eq:relative-source-remainder}
\end{align}

\smallskip
\noindent
{\bf Step~4.} Let us check that   
\begin{equation}\label{xx}
    H_2(U\mid\overline U)
\leq
C\,H_p(U\mid\overline U)\qquad\text{  for any $p>2$},
\end{equation}  whenever $\rho\geq\rho_*$, $|m|<1$, and $0<R_*\leq R(t)\leq R^*<+\infty$. Here the constant $C$ depends only on $p$, $\rho_*$, $R_*$, and $R^*$. 

Indeed, recall the expression \begin{align*}
    H_p(U\mid\overline U) = \rho^p\left[h_p(m)-2\right] + 2\Phi_p(\rho\mid R),
\end{align*}
from which we observe that
\begin{align*}
    h_p(m)-2 &= (1+m)^{1-p} + (1-m)^{1-p}-2\\
&    \geq 2(p-1)\left( \frac{1}{1-m^2}-1 \right)\\
&=(p-1)\left[h_2(m)-2\right].
\end{align*}
Since $\rho\geq\rho_*>0$ and $p>2$, we have $\rho^2
=
\rho^p\rho^{2-p}
\leq
\rho_*^{\,2-p}\rho^p$ and hence 
\begin{equation}
\rho^2\bigl(h_2(m)-2\bigr)
\leq
\frac{\rho_*^{\,2-p}}{p-1}
\rho^p\bigl(h_p(m)-2\bigr).
\label{eq:m-part-H2-Hp}
\end{equation}
On the other hand,  we obtain by Taylor's expansion that 
\begin{equation}
\Phi_p(\rho\mid R)
=
p(p-1)(\rho-R)^2
\int_0^1
(1-\theta)
\bigl(R+\theta(\rho-R)\bigr)^{p-2}
\,d\theta.
\label{eq:Phi-p-integral}
\end{equation}
As $R+\theta(\rho-R)\geq r_*$ for every $\theta \in [0,1]$ for $r_*:=\min\{\rho_*,R_*\}>0$, it follows from \eqref{eq:Phi-p-integral} that
\begin{equation}
2(\rho-R)^2
\leq
\frac{2}{p(p-1)r_*^{\,p-2}}
\,2\Phi_p(\rho\mid R).
\label{eq:density-part-H2-Hp}
\end{equation}
Combining \eqref{eq:m-part-H2-Hp} and
\eqref{eq:density-part-H2-Hp}, we thus conclude that
\begin{align*}
H_2(U\mid\overline U)
&=
\rho^2\bigl(h_2(m)-2\bigr)
+
2(\rho-R)^2
\\
&\leq
C\left[
\rho^p\bigl(h_p(m)-2\bigr)
+
2\Phi_p(\rho\mid R)
\right]
\\
&=
\max\left\{
\frac{\rho_*^{\,2-p}}{p-1},
\frac{2}{
p(p-1)
\bigl(\min\{\rho_*,R_*\}\bigr)^{p-2}
}
\right\} \cdot  H_p(U\mid\overline U). 
\end{align*}
This proves  \eqref{xx}.

\smallskip
\noindent
{\bf Step~5.} Now, let us show that for any $p \in [2,\infty[$, it holds that
\begin{equation} 
\left|
Q_p(U\mid\overline U)
\right|
\leq
C_Q H_p(U\mid\overline U), 
\label{eq:relative-flux-controlled}
\end{equation}
where $C_Q$ depends only on $p$, $\rho_*$, and $R^*$.

For this purpose, denote 
\[
j_p(m)
:=
(m-1)(1+m)^{1-p}
+
(m+1)(1-m)^{1-p}.
\]
Then the relative entropy flux defined in \eqref{eq:relative-entropy-flux} satisfies \begin{align}\label{yy}
Q_p(U\mid\overline U)
&:= q_p(U)-q_p(\overline U)
- \nabla\eta_p(\overline U)
\cdot \bigl(
F(U)-F(\overline U)
\bigr)\nonumber\\
&= 
\rho^{p-1}\bigl(j_p(m)-2pm\bigr)
+
2pm\bigl(\rho^{p-1}-R^{p-1}\bigr).
\end{align}

    For the first term on the right-hand side of \eqref{yy}, we observe that 
\begin{equation}
\left|
j_p(m)-2pm
\right|
\leq
C_p\bigl(h_p(m)-2\bigr)
\qquad\text{ whenever }
|m|<1.
\label{eq:jp-hp-bound}
\end{equation}
Indeed, the quotient $\frac{|j_p(m)-2pm|}{h_p(m)-2}$ extends continuously to the closed interval $[-1,1]$. Near $m=0$, Taylor expansion gives us $j_p(m)-2pm=\mathcal{O} (|m|^3)$ and $
h_p(m)-2=p(p-1)m^2+\mathcal{O}(m^4)$; while near $m=\pm1$, the numerator and denominator  have the same singular order. Thus, as $\rho\geq\rho_*>0$, it follows that
\begin{align}
\rho^{p-1}
\left|
j_p(m)-2pm
\right| \leq
\frac{C_p}{\rho_*}
\rho^p\bigl(h_p(m)-2\bigr).
\label{eq:pure-m-relative-flux}
\end{align}

 For the second term on the right-hand side of \eqref{yy}, by the mean value theorem we have
\begin{equation*}
\left|
\rho^{p-1}-R^{p-1}
\right|
\leq
C_p|\rho-R|(\rho+R)^{p-2}.
\end{equation*}
Moreover, recall that $h_p(m)-2\geq c_pm^2$ and
$\Phi_p(\rho\mid R)
\geq
c_p|\rho-R|^2(\rho+R)^{p-2}$. Also observe that $\rho+R
\leq
\left(
1+\frac{R^*}{\rho_*}
\right)\rho$ as $\rho\geq\rho_*$ and $R\leq R^*$, so \[
\frac{(\rho+R)^{(p-2)/2}}{\rho^{p/2}}
\leq C_{p,\rho_*,R^*}.
\]
We may thus write
\begin{align*}
&|m|\,|\rho-R|(\rho+R)^{p-2}
\\
&\quad=
\left(\rho^{p/2}|m|\right)
\left(
|\rho-R|(\rho+R)^{(p-2)/2}
\right)
\frac{(\rho+R)^{(p-2)/2}}{\rho^{p/2}}
\\
&\quad\leq
C
\left(\rho^{p/2}|m|\right)
\left(
|\rho-R|(\rho+R)^{(p-2)/2}
\right)
\\
&\quad\leq
C\left[
\rho^pm^2+
|\rho-R|^2(\rho+R)^{p-2}
\right].
\end{align*} 
Summarising the estimates in this paragraph, we obtain that \begin{align}
|m|
\left|
\rho^{p-1}-R^{p-1}
\right|
\leq
C\left[
\rho^p\bigl(h_p(m)-2\bigr)
+
\Phi_p(\rho\mid R)
\right].
\label{eq:mixed-relative-flux}
\end{align}

The desired bound~\eqref{eq:relative-flux-controlled} now  follows from \eqref{yy}, \eqref{eq:pure-m-relative-flux}, and
\eqref{eq:mixed-relative-flux}.

\smallskip
\noindent
{\bf Step~6.} Assume momentarily the \emph{claim} that
\begin{equation}
\mathcal R_p
\leq
C H_p(U\mid\overline U).
\label{eq:relative-source-main-estimate}
\end{equation}
In what follows, we show how to conclude the proof from \eqref{eq:relative-source-main-estimate}. 

First, as $\eta_p$ is convex, we have $D^2\eta_p(U)
[
\partial_xU,\partial_xU
]
\geq0$. Integrating the bound~\eqref{eq:relative-entropy-identity} over $\mathbb{R}$ with a suitable cutoff argument detailed below and using
\eqref{eq:relative-source-main-estimate}, we obtain that
\begin{equation}
\frac{d}{dt}
\int_{\mathbb{R}}
H_p(U\mid\overline U)\,\dd x
\leq
C
\int_{\mathbb{R}}
H_p(U\mid\overline U)\,\dd x.
\label{eq:relative-entropy-gronwall-inequality}
\end{equation} 
Therefore, the desired bound~\eqref{eq:global-relative-entropy-bound} follows from Gr\"onwall's inequality. The pointwise bound~\eqref{eq:relative-entropy-controls-Lp}, namely  $H_p(U\mid\overline U)
\geq
c|\rho-R|^p$, implies that \begin{equation*}
  \int_{\mathbb{R}}
|\rho^\varepsilon(t,x)-R(t)|^p\,\dd x
\leq
C
\int_{\mathbb{R}}
H_p\bigl(
U^\varepsilon(t,x)\mid\overline U(t)
\bigr)\,\dd x,  
\end{equation*}
which leads to \eqref{eq:global-density-Lp-bound}.

Now we provide the technical details for the rigorous derivation of \eqref{eq:relative-entropy-gronwall-inequality}. Let us prove that 
\begin{align}
&e^{-C_St}
\int_{\mathbb R}
H_p\bigl(U(t,x)\mid\overline U(t)\bigr)\,\dd x
+
\e\int_0^t e^{-C_Ss}
\int_{\mathbb R}
D^2\eta_p(U)
[\partial_xU,\partial_xU]
\,\dd x\,\dd s
\nonumber\\
&\leq
\int_{\mathbb R}
H_p\bigl(U_0(x)\mid\overline U(0)\bigr)\,\dd x\qquad\text{
for every $t\in[0,T]$,}
\label{eq:global-relative-entropy-cutoff}
\end{align}
where $C=C_S$ is the constant in~\eqref{eq:relative-source-main-estimate}, 
provided that the right-hand side is finite.

To this end, we introduce a family of positive cutoff functions. For $L\geq1$,
set
\begin{equation}
\chi_L(x)
:=
\exp\left(
1-\sqrt{1+\frac{x^2}{L^2}}
\right).
\label{eq:exponential-cutoff}
\end{equation}
Note that $0<\chi_L\leq1$ and $\chi_L(x)\nearrow1$ as $L\to\infty$ for every $x\in\mathbb R$. Furthermore,
\begin{equation}
\frac{\chi_L'(x)}{\chi_L(x)}
=
-\frac{x}{L^2\sqrt{1+x^2/L^2}},\quad |\chi_L'(x)|
\leq
\frac1L\chi_L(x),\quad\text{and  } |\chi_L''(x)|
\leq
\frac{2}{L^2}\chi_L(x).
\label{eq:first-cutoff-derivative}
\end{equation} 

For notational convenience, we put
\begin{align*}
&\mathcal H_p(t,x)
:=
H_p\bigl(U(t,x)\mid\overline U(t)\bigr),\\
&E_L(t)
:=
\int_{\mathbb R}
\chi_L(x)\mathcal{H}_p(t,x)\,\dd x,\\
&D_L(t)
:=
\int_{\mathbb R}
\chi_L(x)
D^2\eta_p(U)
[\partial_xU,\partial_xU]
\,\dd x.
\end{align*}
Then, multiplying~\eqref{eq:relative-entropy-identity} by $\chi_L$
and integrating over $\R$, we obtain that 
\begin{align}
\frac{d}{dt}E_L(t)+\e D_L(t)
={}&
\int_{\mathbb R}
\chi_L'(x)Q_p(U\mid\overline U)\,\dd x
\nonumber\\
&\quad
+
\e\int_{\mathbb R}
\chi_L''(x)\mathcal H_p(t,x)\,\dd x
+
\int_{\mathbb R}
\chi_L(x)\mathcal R_p(t,x)\,\dd x.\label{cutoff, new}
\end{align}
By \eqref{eq:first-cutoff-derivative} and
\eqref{eq:relative-flux-controlled}, we have that
\begin{align*}
\left|
\int_{\mathbb R}
\chi_L'Q_p(U\mid\overline U)\,\dd x
\right|
\leq
\frac{C_Q}{L}E_L(t),\qquad  \e\left|
\int_{\mathbb R}
\chi_L''\mathcal H_p\,\dd x
\right|
\leq
\frac{2\e}{L^2}E_L(t).
\end{align*}
Meanwhile, by \eqref{eq:relative-source-main-estimate} we have 
\begin{equation*}
\int_{\mathbb R}
\chi_L\mathcal R_p\,\dd x
\leq
C_SE_L(t).
\label{eq:cutoff-source-bound}
\end{equation*}
Substituting the previous three estimates into 
\eqref{cutoff, new}, we deduce that
\begin{equation}
\frac{d}{dt}E_L(t)+\e D_L(t)
\leq
a_L E_L(t)
\qquad\text{where }
a_L
:=
C_S+\frac{C_Q}{L}+\frac{2\e}{L^2}.
\label{eq:localized-relative-gronwall}
\end{equation}

Multiplying \eqref{eq:localized-relative-gronwall} by
$e^{-a_Lt}$ yields that
\[
\frac{d}{dt}
\left(
e^{-a_Lt}E_L(t)
\right)
+
\e e^{-a_Lt}D_L(t)
\leq0.
\]
Hence,
\begin{equation}
e^{-a_Lt}E_L(t)
+
\e\int_0^t
e^{-a_Ls}D_L(s)\,\dd s
\leq
E_L(0).
\label{eq:weighted-cutoff-estimate}
\end{equation}
By monotonicity of the cutoff function $\chi_L$ and the monotone convergence theorem, we have
\[
E_L(t)
\longrightarrow
\int_{\mathbb R}
H_p\bigl(U(t,x)\mid\overline U(t)\bigr)\,\dd x
\]
and
\[
D_L(t)
\longrightarrow
\int_{\mathbb R}
D^2\eta_p(U)
[\partial_xU,\partial_xU]\,\dd x
\]
as $L\to\infty$. Moreover, note that $
a_L\rightarrow C_S$ and
$E_L(0)
\rightarrow
\int_{\mathbb R}
H_p\bigl(U_0(x)\mid\overline U(0)\bigr)\,\dd x$. 

We may now conclude \eqref{eq:global-relative-entropy-cutoff} and hence \eqref{eq:relative-entropy-gronwall-inequality} by passing to the limit $L\to\infty$ in
\eqref{eq:weighted-cutoff-estimate}.

\smallskip
\noindent
{\bf Step~7.} Assuming \eqref{eq:relative-source-main-estimate}, it now remains to prove the dissipation estimate~\eqref{rel entropy, gradient estimate}. For this purpose, in view of Step~4 above, it suffices to prove for $p=2$ only. Indeed, by virtue of \eqref{eq:relative-entropy-identity}--\eqref{eq:relative-source-main-estimate} and integration by parts, we have that
\begin{align}\label{new, ineq}
&\int_{\mathbb{R}}
H_2\bigl(U^\e(t)\mid \overline{U}(t)\bigr)\,\dd x
+
\e\int_0^t\int_{\mathbb{R}}
D^2\eta_2\bigl(U^\e(s,x)\bigr)
\bigl[\partial_xU^\e(s,x),\partial_xU^\e(s,x)\bigr]
\,\dd x\,\dd s\nonumber
\\
&\qquad\leq
\int_{\mathbb{R}}
H_2\bigl(U^\e_0(x)\mid \overline{U}(0)\bigr)\,\dd x
+
C\int_0^t\int_{\mathbb{R}}
H_2\bigl(U^\e(s,x)\mid \overline{U}(s)\bigr)
\,\dd x\,\dd s. 
\end{align}
Here, a direct calculation leads to
\[
D^2\eta_2(\rho,m)
=
\begin{bmatrix}
\dfrac{4}{1-m^2}
&
\dfrac{8\rho m}{(1-m^2)^2}
\\
\dfrac{8\rho m}{(1-m^2)^2}
&
\dfrac{4\rho^2(1+3m^2)}{(1-m^2)^3}
\end{bmatrix}.
\]
It follows that
\[
\det D^2\eta_2(\rho,m)
=
\frac{16\rho^2}{(1-m^2)^3}
\]
and hence, as argued in Proposition~\ref{propn: entropy analysis} (2), that
\[
\lambda_{\min}\bigl(D^2\eta_2(\rho,m)\bigr)
\geq
\frac{\det D^2\eta_2(\rho,m)}
     {\operatorname{tr}D^2\eta_2(\rho,m)}
=
\frac{4\rho^2}
     {(1-m^2)^2+\rho^2(1+3m^2)}.
\]
Since $\rho^\e\geq\rho_*>0$ and $|m^\e|<1$, we  have
\begin{equation*}
\lambda_{\min}
\bigl(D^2\eta_2(\rho^\e,m^\e)\bigr)
\geq
\frac{4\rho_*^2}{1+4\rho_*^2}
=:c_*>0.
\end{equation*}
Substituting this bound into~\eqref{new, ineq} and using the boundedness of $\int_0^t\int_{\mathbb{R}}
H_2\bigl(U^\e(s,x)\mid \overline{U}(s)\bigr)
\,\dd x\,\dd s$ established in~\eqref{eq:global-relative-entropy-bound}, we arrive at~\eqref{rel entropy, gradient estimate}.

\smallskip
\noindent
{\bf Step~8.} Finally, we prove the claim~\eqref{eq:relative-source-main-estimate}.

Define the quantity
\begin{align}
G_\rho :=
\partial_\rho\eta_p(U)
-
\partial_\rho\eta_p(\overline U).
\label{eq:rho-entropy-gradient-difference}
\end{align}
Since $\partial_\rho\eta_p(\rho,m)
=
p\rho^{p-1}h_p(m)$, we may express 
\begin{align*}
    G_\rho = p\rho^{p-1}\bigl(h_p(m)-2\bigr)
+
2p\bigl(\rho^{p-1}-R^{p-1}\bigr).
\end{align*}

The first term on the right-hand side has the lower bound:
\begin{equation}
\rho^p\bigl(h_p(m)-2\bigr) \geq C^{-1}\rho^{p-1}
\bigl(h_p(m)-2\bigr)
\bigl(|\rho-R|+|m|\bigr).
\end{equation}
This follows from $|m|\leq1$ and
$\frac{|\rho-R|}{\rho}
\leq
1+\frac{R^*}{\rho_*}$.

Next, for the second term, recall that $$\Phi_p(\rho\mid R)
:=
\rho^p-R^p-pR^{p-1}(\rho-R).$$ We deduce from the mean value theorem
and~\eqref{eq:scalar-relative-entropy-coercivity} that 
\begin{equation*}
\left|
\rho^{p-1}-R^{p-1}
\right|
|\rho-R|
\leq
C\Phi_p(\rho\mid R).
\label{eq:radial-gradient-times-density}
\end{equation*}
On the other hand, 
\begin{equation*}
\left|
\rho^{p-1}-R^{p-1}
\right|
\leq
C|\rho-R|(\rho+R)^{p-2}.
\label{eq:radial-gradient-MVT}
\end{equation*}
Thus, by Young's inequality and the estimate~$(\rho+R)^{p-2}\leq C\rho^p$ (which follows from the positive lower bounds for $\rho, R$), we have that 
\begin{align*}
\left|
\rho^{p-1}-R^{p-1}
\right||m|
&\leq
C|\rho-R|(\rho+R)^{p-2}|m|
\notag\\
&\leq
C|\rho-R|^2(\rho+R)^{p-2}
+
C(\rho+R)^{p-2}m^2
\notag\\
&\leq
C\left[
\Phi_p(\rho\mid R)
+
\rho^pm^2
\right].
\label{eq:radial-gradient-times-momentum}
\end{align*}

Combining the estimates in the previous two paragraphs and invoking the definition of the relative entropy, we deduce the bound \begin{equation}
|G_\rho|
\bigl(
|\rho-R|+|m|
\bigr)
\leq
C H_p(U\mid\overline U)
\label{eq:rho-gradient-key-estimate}
\end{equation}
for $G_\rho$ defined as in \eqref{eq:rho-entropy-gradient-difference}.

Next, observe that 
\[
n(\rho,m)=\frac{m^2-1}{\rho}
\]
has bounded first derivatives in $\rho,m$ (and hence Lipschitz) on
$\left\{
(\rho,m):\rho\geq\rho_*,\ |m|\leq1
\right\}$. Denote
$$S(\rho,m,t)
=
\begin{bmatrix}
-a(t)\rho-2b(t)m-c(t)n\\
-2k(t)m
\end{bmatrix} \equiv \begin{bmatrix}
S^{(1)}(U,t)\\
S^{(2)}(U,t)
\end{bmatrix}.$$

By the assumptions on $a,b,c$ and the Lipschitz boundedness of $n$, as well as recalling the definition of $\overline{U}$ and $H_p(U|\overline{U})$, we have that 
\begin{equation*}
\left|
S^{(1)}(U,t)-S^{(1)}(\overline U,t)
\right|
\leq
C\bigl(
|\rho-R|+|m|
\bigr).
\label{eq:rho-source-lipschitz}
\end{equation*}
This together with~
\eqref{eq:rho-gradient-key-estimate} implies that
\begin{equation}
\left|
G_\rho
\bigl(
S^{(1)}(U,t)-S^{(1)}(\overline U,t)
\bigr)
\right|
\leq
C H_p(U\mid\overline U).
\label{eq:rho-source-relative-estimate}
\end{equation}

For the $m$-component, from the identities $
\partial_m\eta_p(U)=\rho^ph_p'(m)$, $
\partial_m\eta_p(\overline U)=0$, as well as $S^{(2)}(U,t)-S^{(2)}(\overline U,t)
= -2k(t)m$, we deduce that 
\begin{align}
\bigl(
\partial_m\eta_p(U)
-
\partial_m\eta_p(\overline U)
\bigr)
\bigl(
S^{(2)}(U,t)-S^{(2)}(\overline U,t)
\bigr)=
-2k(t)\rho^pmh_p'(m).
\label{eq:m-source-product}
\end{align}
Observe that $h_p'(m)
=
(p-1)
\left[
(1-m)^{-p}-(1+m)^{-p}
\right]$, so $m$ and $h_p'(m)$ have the same sign. This together with  $k(t)\geq0$ leads to
\begin{equation}
-2k(t)\rho^pmh_p'(m)\leq0.
\label{eq:m-source-good-sign}
\end{equation}

It remains to estimate the second line of
\eqref{eq:relative-source-remainder}. Since $S(\overline{U})=[R'(t),0]^\top$, the only nontrivial contribution comes from the $\rho$-component. Set 
\begin{align*}
B_\rho
:=
\partial_\rho\eta_p(U)
-
\partial_\rho\eta_p(\overline U)
-
\partial_{\rho\rho}^2\eta_p(\overline U)(\rho-R)
-
\partial_{\rho m}^2\eta_p(\overline U)m.
\end{align*}  
At the background state $\overline U=[R,0]^\top$ one has that
\begin{align*}
    \partial_{\rho m}^2\eta_p(\overline U)=0\qquad\text{and}\qquad \partial_{\rho\rho}^2\eta_p(\overline U)
= 2p(p-1)R^{p-2}.
\end{align*}
Thus, 
\begin{align*}
B_\rho
=
p\rho^{p-1}\bigl(h_p(m)-2\bigr)+
2p
\left[
\rho^{p-1}
-
R^{p-1}
-
(p-1)R^{p-2}(\rho-R)
\right].
\label{eq:rho-taylor-remainder-explicit}
\end{align*}
By Taylor's expansion, the bound for $\Phi_p$ in ~\eqref{eq:scalar-relative-entropy-coercivity}, as well as the positive lower bounds for $\rho$ and $R$, one may estimate  
\begin{align*}
&\left|
\rho^{p-1}
-
R^{p-1}
-
(p-1)R^{p-2}(\rho-R)
\right|\\& \qquad\leq
C|\rho-R|^2(\rho+R)^{p-3}
\leq
C\Phi_p(\rho\mid R).
\end{align*}
Meanwhile, we have 
\[
\rho^{p-1}\bigl(h_p(m)-2\bigr)
\leq
\frac1{\rho_*}
\rho^p\bigl(h_p(m)-2\bigr).
\]
It then follows that
\begin{equation*}
|B_\rho|
\leq
C H_p(U\mid\overline U).
\label{eq:rho-taylor-remainder-bound}
\end{equation*}
We have also assumed that $R'$ is bounded on $[0,T]$. Therefore, we conclude that
\begin{equation}
\left|
B_\rho R'(t)
\right|
\leq
C H_p(U\mid\overline U).
\label{eq:moving-reference-remainder-bound}
\end{equation}

The \emph{claim}~\eqref{eq:relative-source-main-estimate} now follows from the estimates~\eqref{eq:rho-source-relative-estimate},
\eqref{eq:m-source-good-sign}, and
\eqref{eq:moving-reference-remainder-bound}. This completes the proof of Theorem~\ref{thm:global-relative-entropy}.  \end{proof}

\begin{remark}\label{remark: initial cond R(0)}
For any given $T>0$, we can always choose the initial datum $R(0)$ to ensure the condition $0<R_*\leq R(t)\leq R^*<\infty$ for  $t \in [0,T]$, where $R(t)$ is the solution to the ODE~\eqref{eq:reference-density-ode}. Indeed, $\widetilde{R}:=R^2$ satisfies the equation $  \widetilde{R}'(t) +2a(t)\widetilde{R}(t)=2c(t)$, so the solution is 
\begin{align*}
    {R}(t) = \left\{e^{-2A(t)}\left[{R}_0^2 + 2\int_0^t e^{2A(\tau)}c(\tau)\,\dd \tau\right]\right\}^{\frac{1}{2}},\qquad A(\tau):=\int_0^\tau a(s)\,\dd s.
\end{align*}
Hence, by choosing $R(0)$ sufficiently large (depending on $T$, $\|a\|_{L^\infty([0,T])}$, and $\|c\|_{L^\infty([0,T])}$, and larger than $\frac{2}{Q}$ to ensure that $\overline{U}\big|_{t=0}=[R(0),0]^\top$ lies in the invariant region), we have $0<R_*\leq R(t)\leq R^*<\infty$ for all $t \in [0,T]$. 

\end{remark}

\section{Existence of $W^{2,p}_\loc$-isometric immersions}\label{sec: existence}

We are at the stage of establishing the existence of $W^{2,p}_\loc$-isometric immersions surfaces with $W^{3,\infty}$-Riemannian metrics with negative Gaussian curvature, \emph{provided that} invariant regions exist (Lemma~\ref{lem: Riem inv}) with respect to the PDE~\eqref{rho, m eq, new} for the Riemann invariants $(w^\e,z^\e)$. For this purpose, we first obtain, via the theory of compensated compactness, weak solutions to the Gauss--Codazzi equations~\eqref{eq:Gauss-Codazzi}, and then construct $W^{2,p}_\loc$-isometric immersions via the low-regularity version of the fundamental theorem of surface theory. We shall make crucial use of the uniform spacetime bounds established for the Riemann invariants $(w^\e, z^\e)$, as well as the bounds for $(\rho^\e, m^\e)$ in Proposition~\ref{propn: entropy analysis}.

For convenience of the reader, we recall once again the parabolic system~\eqref{w,z equation; entropy} for the Riemann invariant coordinates:
\begin{equation*}
    \begin{cases}
        \p_tw^\e + z^\e\p_x w^\e = \e \frac{\p_x\left((\rho^\e)^2\p_xw^\e\right)}{(\rho^\e)^2} + w^\e \A_2[w^\e;z^\e]-\A_1[w^\e;z^\e],\\
        \p_tz^\e + w^\e\p_x z^\e = \e \frac{\p_x\left((\rho^\e)^2\p_xz^\e\right)}{(\rho^\e)^2} + z^\e\A_2[w^\e;z^\e]-\A_1[w^\e;z^\e],
    \end{cases}
\end{equation*}
where the source terms are \begin{equation*}
    \A_k[w^\e;z^\e] := \tg^k_{22} + \tg^k_{12}(w^\e+z^\e) + \tg^k_{11}w^\e z^\e,\qquad k \in \{1,2\},
\end{equation*} 
and the symbols $\tg^i_{jk}$ are 
\begin{equation*}
    \begin{cases}
\tg^1_{22} := \G^1_{22},\qquad\tg^1_{12} := \G^1_{12} + \frac{\p_y\gamma}{2\gamma},\qquad\tg^1_{11} := \G^1_{11} + \frac{\p_x \gamma}{\gamma},\\
\tg^2_{22} := \G^2_{22} + \frac{\p_y\gamma}{\gamma},\qquad\tg^2_{12} := \G^2_{12} + \frac{\p_x\gamma}{2\gamma},\qquad\tg^2_{11} := \G^2_{11}.
    \end{cases}
\end{equation*}

The key properties for $\A_1[w^\e;z^\e]$, $\A_2[w^\e;z^\e]$ that shall be used in this section is that they are uniformly bounded in spacetime, provided that their initial data lie in the invariant region; moreover, the bound depends only on the $W^{3,\infty}$-norm of $g$.

%Here and throughout, we identify the time variable $t$ with the spatial coordinate $y$.

\begin{theorem}\label{thm: comp comp}
Let $g$ be a Riemannian metric on $\I_T = [0,T]\times\I\subset \R^2$  with finite $W^{3,\infty}$-norm and strictly negative Gaussian curvature $\gauss$. Assume that the PDE~\eqref{w,z equation; entropy} has solutions $\left(w^\e,z^\e\right)$ whose $L^\infty_t L^\infty_x$-norm on $\I_T$ are bounded independently of $\e$, and satisfy the sign conditions $w^\e\geq 0$, $z^\e \leq 0$. Define the fluid variables $\left(\rho^\e, m^\e\right)$ by
\begin{align*}
          w^\e = \frac{m^\e+1}{\rho^\e},\qquad z^\e  =\frac{m^\e-1}{\rho^\e},
   \end{align*} 
with $\rho^\e$ valued in $[0,\infty]$ for each $\e>0$. In addition, assume for given $p \in [2,\infty[$ that $\{(\rho^\e, m^\e)\}$ are uniformly bounded in $L^p_{t,x}(\I_T)$ and $\{(\sqrt{\e}\p_x\rho^\e, \sqrt{\e}\p_x m^\e)\}$ are uniformly bounded in $L^2_{t,x}(\I_T)$. 

Then, define the geometric variables
\begin{align*}
L^\e:=\gamma \rho^\e, \qquad M^\e = -{\gamma}m^\e,\qquad N^\e=\gamma \left(\frac{(m^\e)^2-1}{\rho^\e}\right).
\end{align*}
There exists a subsequence of $\{(L^\e, M^\e, N^\e)\}$ converging in the weak-$*$ topology of $\left(L^\infty_t L^p_x\right)_\loc(\I_T)$ to some weak solution $\left(\bar{L}, \bar{M}, \bar{N}\right) \in \left(L^\infty_t L^p_x\right)_\loc(\I_T)$ to the Gauss--Codazzi system~\eqref{eq:Gauss-Codazzi}. Furthermore, the metric $g$ admits a $W^{2,p}_\loc$-isometric immersion of $\I_T$ into $(\R^3,\delta)$.
\end{theorem}

To prove Theorem~\ref{thm: comp comp},  recall a compensated compactness framework for the weak continuity of approximate solutions to the Gauss--Codazzi equations. See, for instance, \cite{chen2010a, chen2010c, litzinger2021, chenli2018}.\footnote{In fact, the compensated compactness framework extends to arbitrary dimensions and codimensions; \emph{i.e.}, to the Gauss--Codazzi--Ricci equations.}

\begin{lemma}\label{lem: comp comp framework}
Fix $p \in [2,\infty]$. Consider a family of functions $\left\{(L^\epsilon,M^\epsilon,N^\epsilon)(x,y)\right\}$, defined on an open subset $\Omega\subset\mathbb R^2$, such that:

\begin{enumerate}
\item[(W1)] $\left\{(L^\epsilon,M^\epsilon,N^\epsilon)\right\}$ is bounded in $L^p_\loc(\Omega)$ uniformly in $\e$.

\item[(W2)] $\left\{\p_xM^\epsilon - \p_yL^\epsilon\right\}$ and $\left\{\p_xN^\epsilon - \p_yM^\epsilon\right\}$ are precompact in $W_{\mathrm{loc}}^{-1,2}(\Omega)$.

\item[(W3)] There exist $\sigma_j^\epsilon$, $j=1,2,3$, with $\sigma_j^\epsilon\to 0$ in the sense of distributions as $\epsilon\to 0$ such that
\[
\begin{aligned}
\p_xM^\epsilon - \p_yL^\epsilon &= \Gamma_{22}^2 L^\epsilon - 2\Gamma_{12}^2 M^\epsilon + \Gamma_{11}^2 N^\epsilon + \sigma_1^\epsilon,\\
\p_x N^\epsilon - \p_yM^\epsilon &= -\Gamma_{22}^1 L^\epsilon + 2\Gamma_{12}^1 M^\epsilon - \Gamma_{11}^1 N^\epsilon + \sigma_2^\epsilon,
\end{aligned}
\]
and
\[
L^\epsilon N^\epsilon - (M^\epsilon)^2 = \gauss + \sigma_3^\epsilon.
\]
\end{enumerate}
Then there exists a subsequence (not relabelled) $\left\{(L^\epsilon,M^\epsilon,N^\epsilon)\right\}$ converging weakly in $L^p_\loc$ to $(\bar{L},\bar{M},\bar{N})$ as $\epsilon\to 0$ such that $(\bar{L},\bar{M},\bar{N})$ is a weak solution to the Gauss--Codazzi equations if $p<\infty$. (For $p=\infty$ we take the weak-$*$ convergence of $\left\{(L^\epsilon,M^\epsilon,N^\epsilon)\right\}$.)
\end{lemma}

%The domain $\Omega$ in the above compensated compactness framework may be unbounded. But, for our purposes here, we shall only consider bounded rectangular domains $\I_T$. 

\begin{proof}[Proof of Theorem~\ref{thm: comp comp}]
We divide our arguments into three steps below.

\smallskip
\noindent
{\bf Step~1.} We derive the ``approximate Gauss--Codazzi equations'' for $\left\{(L^\epsilon,M^\epsilon,N^\epsilon)\right\}$.

First, observe that by the definition of $N^\e$, the Gauss equation holds exactly:
\begin{equation}\label{Gauss, eps version}
    L^\epsilon N^\epsilon - (M^\epsilon)^2 = \gauss.
\end{equation}
Next, using the definition of $L^\e$ and $M^\e$ in terms of $(\rho^\e, m^\e)$, the PDE~\eqref{rho, m eq, new}  for $(\rho^\e, m^\e)$, and the definition of $\tg^i_{jk}$, we compute that
\begin{align*}
    \p_x M^\e - \p_t L^\e &= -\p_x \left(\gamma m^\e\right) - \p_t \left(\gamma \rho^\e\right)\\
    &= -\gamma \left\{ \p_t\rho^\e + \p_x m^\e \right\} - m^\e \p_x\gamma - \rho^\e \p_t\gamma\\
    &= -\gamma\left\{\e\p_{xx}\rho^\e -\tg^2_{22}\rho^\e -2 \tg^2_{12}m^\e - \tg^2_{11}\left(\frac{(m^\e)^2-1}{\rho^\e}\right)  \right\}- m^\e \p_x\gamma - \rho^\e \p_t\gamma\\
    &= -\e\gamma\p_{xx}\rho^\e + \gamma \G^2_{22}\rho^\e + 2\gamma \G^2_{12}m^\e + \gamma \G^2_{11}\left(\frac{(m^\e)^2-1}{\rho^\e}\right)\\
    &= -\e\gamma\p_{xx}\rho^\e + \G^2_{22} L^\e -2\G^2_{12} M^\e + \G^2_{11}N^\e.
\end{align*}
Similar computation goes through for $\p_x N^\e - \p_t M^\e$. By identifying $(x,t)=(x,y) \in \Omega$, we obtain the approximate Codazzi equations: 
\begin{align}
    &\p_x M^\e - \p_y L^\e = -\e\gamma\p_{xx}\rho^\e + \G^2_{22} L^\e -2\G^2_{12} M^\e + \G^2_{11}N^\e,\label{codazzi, eps, 1}\\
    &\p_x N^\e - \p_y M^\e =\e\gamma\p_{xx}m^\e - \G^1_{22} L^\e +2\G^1_{12} M^\e - \G^1_{11}N^\e.\label{codazzi, eps, 2}
\end{align}

\smallskip
\noindent
{\bf Step~2.} We verify the assumptions in the compensated compactness framework (Lemma~\ref{lem: comp comp framework}). 

\begin{itemize}
    \item 
Thanks to the $L^\infty_t L^\infty_x$-bound for Riemann invariants $\left(w^\e,z^\e\right)$ and that $\rho^\e \geq 0$, we have $-1 \leq m^\e \leq 1$ and hence $$\left|M^\e\right| \leq \gamma\qquad\text{a.e. on $\I_T$}.$$ Moreover, by Proposition~\ref{propn: entropy analysis}, for each $p \in [2,\infty[$ we may bound the  $L^\infty_t L^p_x$-norm of $\rho^\e$ uniformly in $\e$, with suitably prepared initial-boundary data. Thus, we deduce that $L^\e= \gamma \rho^\e$ and $N^\e = \gamma w^\e z^\e \rho^\e$ are bounded in $L^\infty_t L^p_x$ uniformly in $\e$. This verifies (W1).

From now on, fix one such $p \in [2,\infty[$ in the remaining parts of the proof.

\item 

Recall the following \cite{tartar1979, dcl}:
\begin{lemma*}
Let $\Omega\subset\mathbb R^n$ be an open set for any $n \in \mathbb{N}$. Then
\begin{align*}
&\bigl(\text{compact set of } W_{\mathrm{loc}}^{-1,q}(\Omega)\bigr)
\cap
\bigl(\text{bounded set of } W_{\mathrm{loc}}^{-1,r}(\Omega)\bigr)\\
&\qquad\qquad \subset
\bigl(\text{compact set of } W_{\mathrm{loc}}^{-1,2}(\Omega)\bigr),
\end{align*}
where $q$ and $r$ are constants satisfying $1<q\leq 2<r$.
\end{lemma*}

By Proposition~\ref{propn: entropy analysis} (2), $\{\sqrt{\e}\gamma\p_{x}m^\e\}$ and $\{\sqrt{\e}\gamma\p_{x}\rho^\e\}$ are bounded in $L^2(\I_T)$ uniformly in $\e$. Thus, $\{\e\gamma\p_{xx}\rho^\e\}$ and $\{\e\gamma\p_{xx}m^\e\}$ are compact in $W^{-1,2}(\I_T)$. On the other hand, by (W1) and Sobolev embedding, $\left\{(L^\epsilon,M^\epsilon,N^\epsilon)\right\}$ lies in a compact subset of $W^{-1,q}(\I_T)$ for some $1<q \leq 2$; while by (W1) and Rellich's lemma or Sobolev embedding, $\left\{(L^\epsilon,M^\epsilon,N^\epsilon)\right\}$ is bounded in $W^{-1,r}(\I_T)$ for some $r>2$. Therefore, in view of the lemma quoted above and the approximate Gauss--Codazzi equations~\eqref{Gauss, eps version}, \eqref{codazzi, eps, 1}, and \eqref{codazzi, eps, 2}, (W2) is verified.

\item 
For (W3), by an inspection on~\eqref{Gauss, eps version}, \eqref{codazzi, eps, 1}, and \eqref{codazzi, eps, 2}, we only need to check that $\{\e\gamma\p_{xx}\rho^\e\}$ and $\{\e\gamma\p_{xx}m^\e\}$ tend to zero in the sense of distributions as $\e \to 0$. Indeed, take any test function $\varphi \in C^\infty_c(\I_T)$. We have that
\begin{align*}
\left|\iint_{\I_T} \e\gamma\p_{xx}\rho^\e \varphi\,\dd x\,\dd y\right| &= \sqrt{\e} \left|\iint_{\I_T}\left(\sqrt{\e}\p_x\rho^\e\right) \p_x\left(\gamma\varphi\right)\,\dd x\,\dd y\right|\\
&\leq C\sqrt{\e} \left\| \sqrt{\e}\p_x\rho^\e \right\|_{L^2({\rm spt}(\varphi))}, 
\end{align*}
where $C$ depends only on $\|\gamma\|_{C^1(\I_T)}$, $\|\varphi\|_{C^1(\I_T)}$, and the support of $\varphi$. In light of Proposition~\ref{propn: entropy analysis} (2), we deduce that 
\begin{align*}
\left|\iint_{\I_T}\e\gamma\p_{xx}\rho^\e \varphi\,\dd x\,\dd y\right| \leq C'\sqrt{\e} \longrightarrow 0\quad \text{as } \e \to 0
\end{align*}
for $C'$ depending only on $\|g\|_{W^{3,\infty}(\I_T)}$, $\|\varphi\|_{C^1(\I_T)}$, and the support of $\varphi$. The argument for  $\{\e\gamma\p_{xx}m^\e\}$  is completely analogous. Thus (W3) follows. 
\end{itemize}

\smallskip
\noindent
{\bf Step~3.} By Lemma~\ref{lem: comp comp framework} and Step~2, there exists a subsequence (not relabelled) $\left\{(L^\epsilon,M^\epsilon,N^\epsilon)\right\}$ converging weakly-$*$ in $L^p$ to $(\bar{L},\bar{M},\bar{N})$ as $\epsilon\to 0$ such that $(\bar{L},\bar{M},\bar{N})$ is a weak solution to the Gauss--Codazzi equations. Then, by Lemma~\ref{lem: surface theory}, there exists an essentially unique $W^{2,p}_\loc$-isometric immersion  whose second fundamental form is $\overline{\two}=\begin{bmatrix}
    \bar{L} & \bar{M}\\
    \bar{M} & \bar{N}
\end{bmatrix}$. We may take such $\overline{\two}$ to be defined on the whole strip $\R \times [0,T]$ via a routine exhaustion and diagonalisation argument, and we safely omit the details.

The proof of Theorem~\ref{thm: comp comp} is now complete.  \end{proof}

The key point of Theorem~\ref{thm: comp comp} is to reduce the proof for the existence of isometric immersions to that of the existence of invariant regions for the PDE~\eqref{w,z equation; entropy} satisfied by the Riemann invariant coordinates $\left(w^\e,z^\e\right)$.

%One sufficient condition for the existence of invariant regions for diagonal metrics $g=E(x,y)\,dx^2 + G(y)\,dy^2$ has been established in Lemma~\ref{lem: Riem inv}, which essentially amounts to checking the sign of certain modified Christoffel symbols $\tg^i_{jk}$. Fortunately, many families of negatively curved surface metrics, including those obtained via deforming the metrics for various important classical minimal surfaces, indeed satisfy such conditions. This is the content of the next section. 

\section{Conclusion of Theorems~\ref{thm: main} and \ref{thm: more general metrics}}\label{sec: conclusion}

\begin{proof}[Proof of Theorem~\ref{thm: main}]
We directly verify (see \S\ref{app:metric-calculations}, Table~\ref{table:metric-parameters}) that the metrics listed in Table~\ref{table: 2} satisfy the hypotheses of Theorem~\ref{thm: more general metrics} (2). Thus, Theorem~\ref{thm: main} follows from Theorem~\ref{thm: more general metrics}.   \end{proof}

\begin{proof}[Proof of Theorem~\ref{thm: more general metrics}]

We first show the existence of $W^{2,p}_\loc$-isometric immersions in (1) and (2), and then construct an initial datum for $(\rho_0^\e, m^\e_0)$ such that the corresponding second fundamental form for the isometric immersion is genuinely unbounded near $\{t=0\}$. 

\noindent
\underline{Proof of (1)}. Assume that $g=E(x,y)dx^2+G(y)dy^2$ on $\I_T$, where $T>0$ is an arbitrary positive number and $\I \subset\R$ an arbitrary bounded interval. By definition of the modified Christoffel symbols in~\eqref{modified Gamma}, we have
\begin{align*}
  \tg^1_{22}=-\frac{\p_xG}{2E}=0,\qquad   \tg^1_{12} = \frac{1}{2}\p_y\left[\log (E\gamma)\right].
\end{align*}
For positive $E,G \in W^{3,\infty}(\I_T)$ with $\p_y(E\gamma)\geq 0$ we thus have $\tg^1_{22}=0$ and $\tg^1_{12}\geq 0$, which verify the assumptions of Lemma~\ref{lem: Riem inv}.  

In view of Remark~\ref{remark: initial data}, we may choose $Q_0$ so small, depending only on $T$ and the $W^{3,\infty}(\I_T)$-norm of the metric $g$ (through the constant $K_T$ in the ODE~\eqref{eq:Q-moving-region}), that the lifespan $T_\star$ of \eqref{eq:Q-moving-region} is larger than $10 T$. We thus infer from Lemma~\ref{lem: Riem inv} the existence of time-dependent square-shaped invariant regions $\mathcal{R}(t)$ for each $t \in [0,T]$. That is, the Riemann invariants $(w^\e, z^\e)$, which are solutions to the parabolic system~\eqref{w,z-regularised eq}, have $\linf$-bounds $0 \leq w(t,x) \leq Q(t)$ and $-Q(t) \leq z(t,x) \leq 0$ for \emph{a.e.} $(t,x) \in \I_T$. Here, by solving the ODE~\eqref{eq:Q-moving-region}, we find that $0 < Q(t) \leq C_T$ for each $t \in [0,T]$, where $C_T$ depends only on $T$ and the $W^{3,\infty}(\I_T)$-norm of $g$ (since $C_T=C(Q_0, K_T)$).

Fix any $p \in [2,\infty[$. By Proposition~\ref{propn: entropy analysis} and the boundedness of $\I_T$, we have
\begin{align*}
    \sup_{t \in [0,T]}\int_\I \left[\rho^\e(t,x)\right]^p\,\dd x + \e \iint_{\I_T} \left\{\left|\p_x \rho^\e\right|^2 + \left|\p_x m^\e\right|^2\right\}\,\dd t\,\dd x \leq C\left(p,T,\|g\|_{W^{3,\infty}(\I_T)}\right). 
\end{align*}
We may now conclude from the compensated compactness Theorem~\ref{thm: comp comp} the existence of a $W^{2,p}$-isometric immersion of $(\I_T,g)$ into $(\R^3,\delta)$.

\smallskip
\noindent
\underline{Proof of (2)}. Now we turn to the semiglobal isometric immersions. Assume that $T>0$ is arbitrary and $\I=\R$, and fix any $p \in [2,\infty[$. As in the proof of~(1) above, we may take $Q_0>0$ so small that the square-shaped invariant region $\mathcal{R}(t)$ exists up to time $T$. Here $Q_0$ depends only on $T$, the $W^{3,\infty}$-norm of $g$, and the $\linf$-norm of $\p_y\log\gamma$ (via the $\linf$-norm of $\tg^i_{jk}$, which may fail to be bounded by the $W^{3,\infty}$-norm of $g$ on $\R \times [0,T]$).

Next, we shall apply the relative entropy Theorem~\ref{thm:global-relative-entropy} with 
\begin{align*}
    a=\tg^2_{22},\qquad b=\tg^2_{12},\qquad c=\tg^2_{11},\qquad k=\tg^1_{12}\geq 0.
\end{align*}
Note that $\tg^1_{22}=\tg^1_{11}=0$ when both metric components $E$, $G$ are independent of $x$. The $\linf$-norms of $a$, $b$, $c$,  and $k$ are controlled by the $W^{3,\infty}$-norm of $g$ and the $\linf$-norm of $\p_y\log\gamma$.  The PDE~\eqref{rho, m equation} for the fluid variables $(\rho^\e,m^\e)$ now reduces to the form considered in Theorem~\ref{thm:global-relative-entropy}. The uniform positive bound for $\rho^\e$ follows from the existence of invariant regions --- indeed, by \eqref{eq:rho-lower-bound} one has $\rho^\e \geq  Q^{-1}>0$. 
Moreover, in view of  Remark~\ref{remark: initial cond R(0)}, by choosing the initial datum $R(0)$ for the ODE~\eqref{eq:reference-density-ode} sufficiently large, we can ensure the hypothesis $0<R_*\leq R(t)\leq R^*<\infty$ on the boundedness of the background ODE solution $R(t)$. Here $R(0)$ depends only $T$, the $W^{3,\infty}$-norm of $g$, and the $\linf$-norm of $\p_y\log\gamma$.

By now, in order to apply Theorem~\ref{thm:global-relative-entropy}, it only remains to check that the initial data $(\rho^\e_0, m^\e_0)$ can be chosen such that the $p$-th relative entropy is finite: 
\begin{align}\label{final, condition a}
    \sup_{\e>0}
\int_{\mathbb{R}}
H_p\left(
U^\e_0(x)= [\rho_0^\e, m_0^\e]^\top\mid\overline U(0)
=[R(0),0]^\top\right)\,\dd x
<\infty.
\end{align}
We also require such initial data to be comparable with the conditions in Lemma~\ref{lem: Riem inv}; namely that
\begin{align}\label{final, condition b}
-Q_0 \leq \frac{m^\e_0 -1}{\rho_0^\e} \leq 0 \leq \frac{m^\e_0 + 1}{\rho_0^\e} \leq Q_0.
\end{align}

For the above purpose, we first require $R(0)>{Q_0}^{-1}$ by enlarging $R(0)$ if necessary. We shall select $\rho_0^\e \equiv \rho_0$ and $m_0^\e\equiv m_0$ for all $\e$. In fact, a particular simple choice is 
\[
m_0(x):=0,
\qquad
\rho_0(x):=R_0+f(x),
\]
where $0\leq f \in L^p(\R)$ is of compact support. The corresponding initial Riemann invariants are
\[
w_0(x)
=
\frac{1+m_0(x)}{\rho_0(x)}
=
\frac{1}{R_0+f(x)},\qquad z_0(x)
=
\frac{m_0(x)-1}{\rho_0(x)}
=
-\frac{1}{R_0+f(x)}.
\] 
Since $f\geq0$ and $R_0^{-1}<Q_0$, \eqref{final, condition b} immediately follows.

We next verify~\eqref{final, condition a}. Recall the notation $\Phi_p(\rho\mid R)
:=
\rho^p-R^p-pR^{p-1}(\rho-R)$. As $m_0=0$, the relative entropy then reduces to
\begin{align*}
H_p\left(
U_0\mid\overline U_0
\right)
&=
H_p\left(
[\rho_0,0]^\top
\mid
[R_0,0]^\top
\right)
\nonumber\\
&=
2\Phi_p(\rho_0\mid R_0)=
2\left[
(R_0+f)^p-R_0^p-pR_0^{p-1}f
\right].
\end{align*}
For $f\geq0$, we apply Taylor's expansion to obtain
\[
\Phi_p(R_0+f\mid R_0)
=
p(p-1)f^2
\int_0^1
(1-\theta)(R_0+\theta f)^{p-2}
\,\dd \theta.
\]
Hence,
\begin{equation*}
0
\leq
\Phi_p(R_0+f\mid R_0)
\leq
C_{p,R_0}\bigl(f^2+f^p\bigr).
\end{equation*}
But for $p\geq2$, compactly supported $f\in L^p(\mathbb R)$ also lies in $f\in L^2(\mathbb R)$. Thus~\eqref{final, condition a} follows.

To conclude, we apply Theorem~\ref{thm:global-relative-entropy} to deduce that 
\begin{align*}
&\sup_{t \in [0,T]}\int_\R \left|\rho^\e(t,x)-R(t)\right|^p\,\dd x + \e \iint_{\R \times [0,T]} \left\{\left|\p_x \rho^\e\right|^2 + \left|\p_x m^\e\right|^2\right\}\,\dd t\,\dd x \\
&\qquad \leq C\left(p,T,\|g\|_{W^{3,\infty}(\I_T)}, \|\log\gamma\|_{W^{1,\infty}(\R \times [0,T])},f\right). 
\end{align*}
By construction, $\overline{U}(t)=[R(t),0]^\top$ is a special solution to the PDE~\eqref{rho, m equation} for $(\rho^\e,m^\e)$ with any $\e$. It corresponds to the second fundamental form ${\rm diag}\left(\gamma R(t), -\gamma R^{-1}(t) \right)$ that satisfies the Gauss--Codazzi equations. We may apply the compensated compactness framework in Theorem~\ref{thm: comp comp} to conclude the existence of a $W^{2,p}_\loc$-isometric immersion of $(\R \times [0,T], g)$ into $(\R^3,\delta)$. Also notice that Remark~\ref{remark: special isom imm} follows from this paragraph.

\smallskip
\noindent
\underline{Proof that the second fundamental forms are in $L^p_{\mathrm{loc}}
\setminus L^\infty_{\mathrm{loc}}$.} Denote by $\overline \two=\begin{bmatrix}
    L&M\\
    M&N
\end{bmatrix}$ the second fundamental form of the $W^{2,p}_\loc$-isometric immersions, whose existence has been established in the earlier parts of the proof. It suffices to prove for the case $\I=\R$ only; in fact, let us prove that $\overline \two$ lies in $L^p_\loc(\R \times [0,T])$ for $2\leq p<\infty$ but not in $L^\infty_\loc(\R\times [0,\delta_0])$ for any $\delta_0>0$.

Fix any nonnegative $\chi\in C_c^\infty(\mathbb R)$ such that
$\chi\equiv1$ in a neighbourhood of the origin. In the proof of (2) above, take 
\begin{align*}
    f(x)
=
\chi(x)|x|^{-\alpha},
\qquad
0<\alpha<\frac1p.
\end{align*}
The conditions~\eqref{final, condition a} and \eqref{final, condition b} are compatible with an unbounded initial density:
\[
\rho_0-R_0=f\in L^p(\mathbb R)\setminus L^\infty(\mathbb R).
\]
Thus, if $\gamma(x,0) \geq \gamma_\star >0$, then the  corresponding component of the initial second fundamental
form $L_0(x)=\gamma(x,0)\rho_0(x)$ lies in $L_0\in L^p_{\mathrm{loc}}(\mathbb R)
\setminus L^\infty_{\mathrm{loc}}(\mathbb R)$.

Let us label $$U_0
=
(\rho_0,m_0)^\top
=
(R_0+f,0)^\top$$ and establish
\begin{equation}
U(t,\cdot)
\longrightarrow
U_0
\qquad
\text{strongly in }L^p_{\loc}(\R)
\quad\text{as }t\to0^+.
\label{eq:strong-initial-trace}
\end{equation}
Once this is proved, we know that the second fundamental form component $L \in L^p_\loc (\R \times [0,\delta_0]) \setminus L^\infty_\loc(\R \times [0,\delta_0])$ for any $\delta_0>0$.

\begin{proof}[Proof of the strong convergence of initial trace~\eqref{eq:strong-initial-trace}]
We divide the arguments into six steps.

\smallskip
\noindent
\textbf{Step 1.}
Choose $0 \leq f^\e\in C_c^\infty(\R)$ such that $f^\e \rightarrow f$ strongly in $L^p(\R)\cap L^2(\R)$. This is possible for $p\geq2$ and $f$ compactly supported. Set
\[
U_0^\e
=
(\rho_0^\e,m_0^\e)^\top
:=
(R_0+f^\e,0)^\top.
\]
Then $U_0^\e\rightarrow U_0$
strongly in $L^p_{\loc}(\R)$. Moreover, since $m_0^\e=0$ we have that
\[
H_p(U_0^\e\mid\overline U(0))
=
2\Phi_p(R_0+f^\e\mid R_0),
\]
where $\Phi_p(\rho\mid R) :=
\rho^p-R^p-pR^{p-1}(\rho-R)$ as before. The bound (for $s\geq 0$)
\[
0\leq
\Phi_p(R_0+s\mid R_0)
\leq
C_{p,R_0}(s^2+s^p)
\]
and the strong convergence of $f^\e$ in $L^p\cap L^2$ imply that
\begin{equation}
\int_{\R}
H_p(U_0^\e\mid\overline U(0))\,\dd x
\longrightarrow
\int_{\R}
H_p(U_0\mid\overline U(0))\,\dd x.
\label{eq:initial-relative-entropy-convergence}
\end{equation}

\smallskip
\noindent
\textbf{Step 2.} We first prove the weak convergence version of \eqref{eq:strong-initial-trace}, \emph{i.e.}, replacing the strong $L^p$-convergence by weak $L^p$-convergence therein. Indeed, set
\begin{align*}
   V(t,x)
:=
U(t,x)-\overline U(t)
=
\begin{bmatrix}
\rho(t,x)-R(t)\\
m(t,x)
\end{bmatrix},\qquad  V_0(x)
:=
U_0(x)-\overline U(0)
=
\begin{bmatrix}
f(x)\\
0
\end{bmatrix}.
\end{align*}
By the coercivity of relative entropy (see~\eqref{eq:relative-entropy-controls-Lp} and its proof), as $\rho\geq\rho_*>0$ and $|m|\leq1$ we have
\[
H_p(U\mid\overline U)
\geq
c\left(
|\rho-R|^p+|m|^p
\right).
\]
Thus $V\in L^\infty(0,T;L^p(\R))$. 
Since $\overline U$ is a spatially homogeneous solution of the
balance law, 
\begin{equation}
\partial_tV
+
\partial_x
\left(
F(U)-F(\overline U)
\right)
=
S(U)-S(\overline U)
\label{eq:equation-for-V}
\end{equation}
in the sense of distributions, where $F(\rho,m)
=\left[m,\,\dfrac{m^2-1}{\rho}\right]^\top$. The function $n(\rho,m):=\frac{m^2-1}{\rho}$ has bounded first derivatives on $\{(\rho,m):\rho\geq\rho_*,\ |m|\leq1\}$. Thus, 
\begin{equation}
\left|
F(U)-F(\overline U)
\right|
+
\left|
S(U)-S(\overline U)
\right|
\leq
C\left(
|\rho-R|+|m|
\right).
\label{eq:flux-source-lipschitz-trace}
\end{equation}
Then for every
$\varphi\in C_c^\infty(\R;\R^2)$ it holds that 
\begin{align*}
\int_{\R}
\bigl(V(t,x)-V_0(x)\bigr)\cdot\varphi(x)\,\dd x
={}&
\int_0^t\int_{\R}
\bigl(F(U)-F(\overline U)\bigr)
\cdot\partial_x\varphi\,\dd x\,\dd s
\\
&+
\int_0^t\int_{\R}
\bigl(S(U)-S(\overline U)\bigr)
\cdot\varphi\,\dd x\,\dd s,
\end{align*}
where the right-hand side tends to zero as $t \to 0^+$ by \eqref{eq:flux-source-lipschitz-trace}. Thus, \begin{equation}
V(t,\cdot)
\rightharpoonup
V_0
\qquad
\text{weakly in }L^p(\R)
\quad\text{as }t\to0^+.
\label{eq:weak-initial-trace-V}
\end{equation}
Moreover, it follows from \eqref{eq:equation-for-V} that $
\partial_tV$ is bounded in $L^\infty(0,T;W^{-1,p}(\R))$.
Hence, after taking a weakly continuous representative, $
V\in C_{\mathrm w}([0,T];L^p(\R))$. 
 
 Note that \eqref{eq:weak-initial-trace-V} is equivalent to $\rho(t,\cdot)-R(t)
\rightharpoonup f$ and $m(t,\cdot)
\rightharpoonup0$ weakly in $L^p(\R)$. In particular, $\rho(t,\cdot)
\rightharpoonup
\rho_0$ weakly in $L^p_{\loc}(\R)$.  

\smallskip
\noindent
\textbf{Step 3.} Define for \emph{a.e.} $t\in[0,T]$ that
\[
E_p(t)
:=
\int_{\R}
H_p\bigl(U(t,x)\mid \overline U(t)\bigr)\,\dd x.
\]  
Recall \eqref{eq:relative-entropy-gronwall-inequality} (so $E_p(t) \leq e^{ct}E_p(0)$ by Gr\"{o}nwall) in the proof of Theorem~\ref{thm:global-relative-entropy}; it implies that $\limsup_{t\to0^+}E_p(t)
\leq E_p(0)$. On the other hand, the weak convergence
\eqref{eq:weak-initial-trace-V}, the convergence $R(t)\to R_0$, and the convexity of the relative entropy imply that $
E_p(0)
\leq
\liminf_{t\to0^+}E_p(t)$. Hence,\footnote{The relative-entropy inequality is initially obtained for almost
every $t$. The conclusion above extends to the weakly continuous
representative: for an exceptional time, one approximates it by
Lebesgue times and uses weak continuity together with weak lower semicontinuity of the relative entropy.}
\begin{equation}
E_p(t)
\longrightarrow
E_p(0)
\qquad
\text{as }t\to0^+.
\label{eq:relative-entropy-converges-at-zero}
\end{equation}

\smallskip
\noindent
\textbf{Step 4.} Recall the notation $h_p(m)
:=
(1+m)^{1-p}+(1-m)^{1-p}$. Let us write 
\begin{align*}
    E_p(t) &= A_p(t)+B_p(t)\\
    &:= \int_{\R}
\rho(t,x)^p
\bigl(h_p(m(t,x))-2\bigr)\,\dd x + 
2\int_{\R}
\Phi_p(\rho(t,x)\mid R(t))\,\dd x.
\end{align*}
Note that $A_p(0)=0$ and hence $E_p(0)=B_p(0)$ since $m_0=0$. By the convexity of $\rho\mapsto\Phi_p(\rho\mid R)$ and that $\rho(t,\cdot)
\rightharpoonup
\rho_0$ weakly in $L^p_{\loc}(\R)$ (see Step~2), it holds that $B_p(0)
\leq
\liminf_{t\to0^+}B_p(t)$. But $
0\leq B_p(t)\leq E_p(t)$, so by \eqref{eq:relative-entropy-converges-at-zero} we have that $B_p(0)
\leq
\liminf_{t\to0^+}B_p(t)
\leq
\limsup_{t\to0^+}B_p(t)
\leq
E_p(0)
=
B_p(0)$. Therefore,
\begin{equation}
B_p(t)\longrightarrow B_p(0)
\qquad
\text{as }t\to0^+.
\label{eq:Bp-convergence}
\end{equation}
It then follows from
$A_p(t)=E_p(t)-B_p(t)$ that
\begin{equation}
A_p(t)\longrightarrow0
\qquad
\text{as }t\to0^+.
\label{eq:Ap-convergence}
\end{equation}

Furthermore, we remark that the convergence result~\eqref{eq:relative-entropy-converges-at-zero} can be localised: Let $K\Subset\R$ be a bounded interval and set
\[
B_p(t;K)
:=
2\int_K
\Phi_p(\rho(t,x)\mid R(t))\,\dd x.
\]
By weak lower semicontinuity and~\eqref{eq:Bp-convergence}, one has
\begin{align*}
\limsup_{t\to0^+}B_p(t;K)
&=
\limsup_{t\to0^+}
\left(
B_p(t)-B_p(t;\R\setminus K)
\right)
\\
&\leq
B_p(0)-B_p(0;\R\setminus K)
=
B_p(0;K).
\end{align*}
It thus follows that
\begin{equation}
B_p(t;K)
\longrightarrow
B_p(0;K)
\qquad
\text{as }t\to0^+.
\label{eq:local-Bp-convergence}
\end{equation}

\smallskip
\noindent
\textbf{Step 5.} Now we are ready to conclude. From the definition of $B_p$, we deduce the identity:
\begin{align*}
\frac12B_p(t;K)
={}&
\int_K\rho(t,x)^p\,\dd x
-
|K|R(t)^p-
pR(t)^{p-1}
\left(
\int_K\rho(t,x)\,\dd x-|K|R(t)
\right).
\end{align*}
By the weak $L^p_{\loc}(\R)$-convergence $\rho(t,\cdot)
\rightharpoonup
\rho_0$, $R(t)\to R_0$, and~\eqref{eq:local-Bp-convergence}, we have 
\begin{equation}
\int_K\rho(t,x)^p\,\dd x
\longrightarrow
\int_K\rho_0(x)^p\,\dd x.
\label{eq:rho-p-norm-convergence}
\end{equation}
Thus, the Radon--Riesz property of $L^p(K)$ and the arbitrariness of $K\Subset \R$ give us 
\begin{equation}\label{eq:rho-strong-trace}
\rho(t,\cdot)
\longrightarrow
\rho_0
\qquad
\text{strongly in }L^p_{\loc}(\R).
\end{equation}

On the other hand, the function $h_p$ is even and strictly convex, with
$h_p(0)=2$ and $h_p'(0)=0$. Hence, 
\[
h_p(m)-2\geq c_p m^2
\qquad\text{for } |m|<1.
\]
Using also $\rho\geq\rho_*>0$, we obtain that $A_p(t)
\geq
c_p\rho_*^p
\int_{\R}|m(t,x)|^2\,\dd x$. It follows from \eqref{eq:Ap-convergence} that $m(t,\cdot)
\longrightarrow0$ strongly in $L^2(\R)$. But $p\geq2$ and $|m|\leq1$, we have $|m|^p\leq |m|^2$. Hence,
\begin{equation}
m(t,\cdot)
\longrightarrow0
\qquad
\text{strongly in }L^p(\R).
\label{eq:m-strong-trace}
\end{equation}

The desired conclusion~\eqref{eq:strong-initial-trace} follows from \eqref{eq:rho-strong-trace} and
\eqref{eq:m-strong-trace}.  \end{proof}

\smallskip

This concludes the proof of Theorem~\ref{thm: more general metrics}.  \end{proof}

  \begin{appendix}

\section{Explicit curvature and monotonicity calculations}
\label{app:metric-calculations}

The parameters $\gauss$, $\gamma := \sqrt{-\gauss}$, and $\tg^1_{22}$ are presented in the following:

\begingroup
\scriptsize
\setlength{\tabcolsep}{2pt}
\renewcommand{\arraystretch}{1.35}

\begin{longtable}{
    @{}
    >{\centering\arraybackslash}p{6cm}
    >{\centering\arraybackslash}p{3cm}
    >{\centering\arraybackslash}p{3cm}
    >{\centering\arraybackslash}p{3cm}
    @{}
}

\caption{Parameters for the
fifteen families of metrics in Table~\ref{table: 2}}
\label{table:metric-parameters}
\\

\toprule
\textbf{Metric}
&
\(\boldsymbol{\gamma}\)
&
\(\boldsymbol{\kappa_\Sigma}\)
&
\(\displaystyle
 \boldsymbol{\frac12\partial_y\log(E\gamma)}\)
\\
\midrule
\endfirsthead

\multicolumn{4}{c}{
\tablename~\thetable\ (continued)
}
\\
\toprule
\textbf{Metric}
&
\(\boldsymbol{\gamma}\)
&
\(\boldsymbol{\kappa_\Sigma}\)
&
\(\displaystyle
 \boldsymbol{\frac12\partial_y\log(E\gamma)}\)
\\
\midrule
\endhead

\midrule
\multicolumn{4}{r}{\emph{Continued on the next page}}
\\
\endfoot

\bottomrule
\endlastfoot

% 1
\(
\displaystyle
g=P(y)\,dx^2+dy^2,
\quad
P(y)=ay^2+by+c,
\quad
a>0,
\quad
\Delta:=4ac-b^2>0
\)
&
\(
\displaystyle
\frac{\sqrt{\Delta}}{2P(y)}
\)
&
\(
\displaystyle
-\frac{\Delta}{4P(y)^2}
\)
&
\(
\displaystyle
0
\)
\\
\midrule

% 2
\(
\displaystyle
g=e^{ay}\,dx^2+dy^2,
\quad a>0
\)
&
\(
\displaystyle
\frac{a}{2}
\)
&
\(
\displaystyle
-\frac{a^2}{4}
\)
&
\(
\displaystyle
\frac{a}{2}
\)
\\
\midrule

% 3
\(
\displaystyle
g=y^q\,dx^2+dy^2,
\quad q>2
\)
&
\(
\displaystyle
\frac{\sqrt{q(q-2)}}{2y}
\)
&
\(
\displaystyle
-\frac{q(q-2)}{4y^2}
\)
&
\(
\displaystyle
\frac{q-1}{2y}
\)
\\
\midrule

% 4
\(
\displaystyle
g=\cosh(ay)\,dx^2+dy^2,
\quad a>0
\)
&
\(
\displaystyle
\frac{
a\sqrt{1+\cosh^2(ay)}
}{
2\cosh(ay)
}
\)
&
\(
\displaystyle
-\frac{
a^2\bigl(1+\cosh^2(ay)\bigr)
}{
4\cosh^2(ay)
}
\)
&
\(
\displaystyle
\frac{
a\sinh(ay)\cosh(ay)
}{
2\bigl(1+\cosh^2(ay)\bigr)
}
\)
\\
\midrule

% 5
\(
\displaystyle
g=(a-y)^{-b}\,dx^2+dy^2,
\quad a,b>0
\)
&
\(
\displaystyle
\frac{
\sqrt{b(b+2)}
}{
2(a-y)
}
\)
&
\(
\displaystyle
-\frac{
b(b+2)
}{
4(a-y)^2
}
\)
&
\(
\displaystyle
\frac{
b+1
}{
2(a-y)
}
\)
\\
\midrule

% 6
\(
\displaystyle
g=(\sinh y+c)\,dx^2+dy^2,
\quad c>1; \mathcal D_c(y)
:=
\sinh^2y+2c\sinh y-1
\)
&
\(
\displaystyle
\frac{
\sqrt{\mathcal D_c(y)}
}{
2(\sinh y+c)
}
\)
&
\(
\displaystyle
-\frac{
\mathcal D_c(y)
}{
4(\sinh y+c)^2
}
\)
&
\(
\displaystyle
\frac{
(\sinh y+c)\cosh y
}{
2\mathcal D_c(y)
}
\)
\\
\midrule

% 7
\(
\displaystyle
g=\cosh^a y\,dx^2+dy^2,
\quad a\geq\frac23
\)
&
\(
\displaystyle
\frac{
\sqrt{a\bigl(2+a\sinh^2y\bigr)}
}{
2\cosh y
}
\)
&
\(
\displaystyle
-\frac{
a\bigl(2+a\sinh^2y\bigr)
}{
4\cosh^2y
}
\)
&
\(
\displaystyle
\frac{
\sinh y\,
\bigl(a^2\sinh^2y+3a-2\bigr)
}{
2\cosh y\bigl(2+a\sinh^2y\bigr)
}
\)
\\
\midrule

% 8
\(
\displaystyle
g=y^2(1+y^2)^{2\beta}\,dx^2
 +(1+y^2)^{2\beta}\,dy^2,
\quad \beta>0
\)
&
\(
\displaystyle
2\sqrt{\beta}\,
(1+y^2)^{-\beta-1}
\)
&
\(
\displaystyle
-4\beta
(1+y^2)^{-2\beta-2}
\)
&
\(
\displaystyle
\frac{
1+\beta y^2
}{
y(1+y^2)
}
\)
\\
\midrule

% 9
\(
\displaystyle
g=e^{ay}\,dx^2+e^{-ay}\,dy^2,
\quad a>0
\)
&
\(
\displaystyle
\frac{a}{\sqrt{2}}e^{ay/2}
\)
&
\(
\displaystyle
-\frac{a^2}{2}e^{ay}
\)
&
\(
\displaystyle
\frac{3a}{4}
\)
\\
\midrule

% 10
\(
\displaystyle
g=A\cosh(\omega y)\,dx^2
 +\bigl(A\cosh(\omega y)\bigr)^{-1}\,dy^2,
\quad A,\omega>0
\)
&
\(
\displaystyle
\omega
\sqrt{
\frac{A\cosh(\omega y)}{2}
}
\)
&
\(
\displaystyle
-\frac{A\omega^2}{2}
\cosh(\omega y)
\)
&
\(
\displaystyle
\frac{3\omega}{4}
\tanh(\omega y)
\)
\\
\midrule

% 11
\(
\displaystyle
g=(1+y^2)\,dx^2
 +\frac{dy^2}{1+y^2}
\)
&
\(
\displaystyle
1
\)
&
\(
\displaystyle
-1
\)
&
\(
\displaystyle
\frac{y}{1+y^2}
\)
\\
\midrule

% 12
\(
\displaystyle
g=e^{ay^q}(dx^2+dy^2),
\quad
a>0,
\quad
1<q\leq2
\)
&
\(
\displaystyle
\sqrt{
\frac{aq(q-1)}{2}
}\,
y^{(q-2)/2}
e^{-ay^q/2}
\)
&
\(
\displaystyle
-\frac{aq(q-1)}{2}\,
y^{q-2}
e^{-ay^q}
\)
&
\(
\displaystyle
\frac{
aqy^q+q-2
}{
4y
}
\)
\\
\midrule

% 13
\(
\displaystyle
g=e^{ay^2+by}(dx^2+dy^2),
\quad a,b>0
\)
&
\(
\displaystyle
\sqrt{a}\,
e^{-(ay^2+by)/2}
\)
&
\(
\displaystyle
-ae^{-(ay^2+by)}
\)
&
\(
\displaystyle
\frac{2ay+b}{4}
\)
\\
\midrule

% 14
\(
\displaystyle
g=e^{a\sinh y}(dx^2+dy^2),
\quad a>0
\)
&
\(
\displaystyle
\sqrt{
\frac{a\sinh y}{2}
}\,
e^{-a\sinh y/2}
\)
&
\(
\displaystyle
-\frac{a}{2}
\sinh y\,
e^{-a\sinh y}
\)
&
\(
\displaystyle
\frac14
\bigl(
a\cosh y+\coth y
\bigr)
\)
\\
\midrule

% 15
\(
\displaystyle
g=\cosh^a y\,(dx^2+dy^2),
\quad a\geq2
\)
&
\(
\displaystyle
\sqrt{\frac{a}{2}}\,
\cosh^{-(a+2)/2}y
\)
&
\(
\displaystyle
-\frac{a}{2}
\cosh^{-a-2}y
\)
&
\(
\displaystyle
\frac{a-2}{4}
\tanh y
\)
\\

\end{longtable}
\endgroup

  \end{appendix}

\bigskip

\noindent
{\bf Acknowledgement}. The author thanks Professors~Raz Kupferman and Gui-Qiang Chen for very insightful discussions. The research of SL is supported by NSFC Projects 12331008 $\&$ 12411530065, the Young Elite Scientists Sponsorship Program by CAST 2023QNRC001, National Key Research $\&$ Development Programs 2023YFA1010900 and 2024YFA1014900, Shanghai Rising-Star Program 24QA2703600, Shanghai Qi-Guang Scholarship, and Shanghai Frontiers Science Center of Modern Analysis.

\medskip
\noindent
{\bf Statement of competing interests}. We declare that there are no conflicts of interest involved.

\medskip
\noindent
{\bf Data Availability Statement}. We declare that no data are associated with this work.

\medskip
\noindent
{\bf AI Statement}.  The author thanks the DeepSeek and ChatGPT 5.5/5.6 AI models for computational assistance during the exploration of geometric conditions and the construction of counterexamples for the metric families considered in this work. All mathematical derivations, conclusions, and errors remain solely the responsibility of the author.

\end{document}